\documentclass{amsart}

\usepackage{mathtools}
\usepackage{xcolor}
\usepackage{tikz}
\usepackage{float}
\usepackage{enumerate}
\usepackage[hidelinks]{hyperref}
\usetikzlibrary{arrows.meta,calc,decorations.pathreplacing}

\providecolor{PurdueGold}{HTML}{CEB888}
\providecolor{PurdueBlack}{HTML}{111111}
\providecolor{DeepBlue}{HTML}{17324D}
\providecolor{AccentRed}{HTML}{8A1538}

\newcommand{\C}{\mathbb C}
\newcommand{\R}{\mathbb R}
\newcommand{\N}{\mathbb N}

\newcommand{\U}{\mathcal U}
\newcommand{\tr}{\operatorname{tr}}
\newcommand{\rank}{\operatorname{rank}}

\newcommand{\supp}{\operatorname{supp}}

\newcommand{\rr}{\operatorname{rr}}

\theoremstyle{plain}
\newtheorem{thm}{Theorem}[section]
\newtheorem{prop}[thm]{Proposition}
\newtheorem{lem}[thm]{Lemma}
\newtheorem{cor}[thm]{Corollary}

\theoremstyle{definition}
\newtheorem{defn}[thm]{Definition}
\newtheorem{notation}[thm]{Notation}

\newtheorem{rmk}[thm]{Remark}

\title{Subquadratic growth and uniform property \(\Gamma\)}

\author{Ethan Kessinger}
\address{Department of Mathematics\\
Purdue University\\
150 North University Street\\
West Lafayette, IN 47907, USA}
\email{ekessin@purdue.edu}

\author{Andrew S. Toms}
\address{Department of Mathematics\\
Purdue University\\
150 North University Street\\
West Lafayette, IN 47907, USA}
\email{atoms@purdue.edu}

\date{}

\hypersetup{
  pdftitle={Subquadratic growth and uniform property Gamma},
  pdfauthor={Ethan Kessinger and Andrew S. Toms}
}

\begin{document}

\begin{abstract}
We prove that every unital separable ASH algebra with subquadratic growth has uniform property $\Gamma$ whenever it has no nonzero finite-dimensional representations.  When simple and non-elementary, these algebras therefore satisfy the Toms--Winter regularity conjecture despite the fact that they generally fail its three conjecturally equivalent properties.  In light of the second author's recent construction of a unital simple separable AH algebra of quadratic growth which fails uniform property $\Gamma$, we conclude that the quadratic dimension growth scale (equivalently, the 2-norm slow dimension growth scale) is the precise geometric threshold governing the potential failure of uniform property $\Gamma$.  We also extend recent work of Elliott--Niu and Vaccaro to the optimal
subquadratic scale by proving that separable unital
\(C^*\)-algebras with locally tracially subquadratic
RSH approximation have uniform property \(\Gamma\), provided that they have no nonzero finite-dimensional representations.
\end{abstract}

\maketitle

\section{Introduction}

Uniform property $\Gamma$ is a central tracial divisibility property for C$^*$-algebras modeled on property $\Gamma$ from the theory of $\mathrm{II}_1$ factors \cite{CETWW:IM}, \cite{MurrayVonNeumann4}.  It is deeply entwined with the Toms--Winter regularity conjecture for simple separable nuclear C$^*$-algebras---the conjecture simply holds in its presence \cite[Theorem~5.6]{CETW:IMRN}. The second author's recent construction of a simple unital separable AH algebra without uniform property $\Gamma$ shows that the presence of uniform property $\Gamma$ in simple separable nuclear C$^*$-algebras is not automatic and that its failure appears to depend on the presence of quadratic dimension growth \cite[Theorem~6]{TomsSchubCalc:preprint}.  Here we prove the natural counterpoint to that construction, namely, that subquadratic growth implies uniform property $\Gamma$.  This solidifies quadratic dimension growth or, more generally, the presence or absence of uniform property $\Gamma$, as a new and fundamental geometric threshold in the theory of nuclear C$^*$-algebras.

For an accessible introduction to dimension growth we restrict our initial attention to AH algebras.  We will extend these notions to ASH algebras in the sequel.  Let $A = \varinjlim (A_i,\phi_i)$ be a unital AH decomposition, so that
\[
A_i = \bigoplus_{j=1}^{k_i} \ p_{i,j}(C(X_{i,j}) \otimes \mathcal{K})p_{i,j},
\]
where $X_{i,j}$ is a compact Hausdorff space and $p_{i,j} \in C(X_{i,j}) \otimes \mathcal{K}$ is a constant-rank projection.  Set
\[
R_i = \max_j \left\{ \frac{\dim(X_{i,j})}{\rank(p_{i,j})^2} \right\}.
\]
We say the decomposition $(A_i,\phi_i)$ has \emph{subquadratic growth} if $\liminf R_i = 0$, and that $A$ has subquadratic growth if it admits such a decomposition.

\begin{figure}[H]
\centering
\begin{tikzpicture}[scale=0.9, every node/.style={font=\small}]
  \def\h{0.18}
  \coordinate (L) at (0.45,0);
  \coordinate (A) at (2.15,0);
  \coordinate (B) at (4.20,0);
  \coordinate (C) at (6.95,0);
  \coordinate (D) at (8.85,0);
  \coordinate (R) at (11.85,0);

  \draw[thick, {Latex[length=3mm]}-{Latex[length=3mm]}] (0.0,0) -- (12.25,0);

  \fill[DeepBlue!25] ($(L)+(0,-\h)$) rectangle ($(A)+(0,\h)$);
  \fill[DeepBlue!50] ($(A)+(0,-\h)$) rectangle ($(B)+(0,\h)$);
  \fill[PurdueGold!45] ($(B)+(0,-\h)$) rectangle ($(C)+(0,\h)$);
  \fill[PurdueGold!80!black!12] ($(C)+(0,-\h)$) rectangle ($(D)+(0,\h)$);
  \fill[AccentRed!18] ($(D)+(0,-\h)$) rectangle ($(R)+(0,\h)$);

  \draw[thick, PurdueBlack!70] ($(L)+(0,-\h)$) rectangle ($(R)+(0,\h)$);
  \draw[PurdueBlack!70] ($(A)+(0,-\h)$) -- ($(A)+(0,\h)$);
  \draw[PurdueBlack!70] ($(B)+(0,-\h)$) -- ($(B)+(0,\h)$);
  \draw[PurdueBlack!70] ($(C)+(0,-\h)$) -- ($(C)+(0,\h)$);
  \draw[PurdueBlack!70] ($(D)+(0,-\h)$) -- ($(D)+(0,\h)$);

  \node[above=5pt] at ($(L)!0.5!(A)+(0,\h)$) {slow};
  \node[above=5pt] at ($(A)!0.5!(B)+(0,\h)$) {linear};
  \node[above=5pt] at ($(B)!0.5!(C)+(0,\h)$) {superlinear};
  \node[above=5pt] at ($(C)!0.5!(D)+(0,\h)$) {quadratic};
  \node[above=5pt] at ($(D)!0.5!(R)+(0,\h)$) {superquadratic};

  \draw[decorate, decoration={brace, mirror, amplitude=6pt}]
    ($(L)+(0,-\h)$) -- ($(C)+(0,-\h)$)
    node[midway, below=10pt, align=center] {subquadratic/2-norm slow growth \\ uniform $\Gamma$ holds};

  \draw[decorate, decoration={brace, mirror, amplitude=6pt}]
    ($(C)+(0,-\h)$) -- ($(R)+(0,-\h)$)
    node[midway, below=10pt, align=center, text width=3.8cm]
    {quadratic or faster/ \\2-norm linear or faster \\ uniform $\Gamma$ may fail};
\end{tikzpicture}
\caption{A coarse taxonomy of geometric dimension-growth regimes.}
\label{fig:geometric-thresholds}
\end{figure}
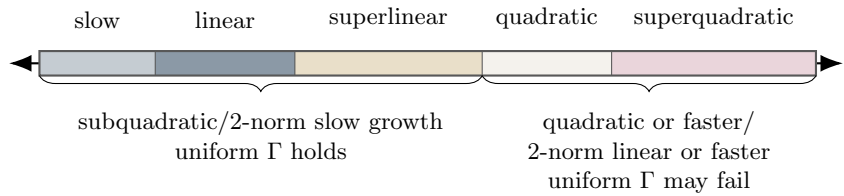

The sequel is devoted essentially to proving the statement under left-hand brace in Figure~\ref{fig:geometric-thresholds} (borrowed here from \cite{TomsSchubCalc:preprint}), not only for AH algebras but also for the larger class of ASH algebras without finite-dimensional representations.
The mechanics of the subquadratic/quadratic boundary are already visible in the geometry of self-adjoint matrices.  For fixed \(\lambda\in\mathbb R\),
the matrices in \(M_r(\mathbb C)_{\mathrm{sa}}\) for which \(\lambda\)
has multiplicity at least \(k\) form a closed stratified subset whose
largest stratum has real codimension \(k^2\).  A self-adjoint section
over a base space of dimension \(d<k^2\) can therefore be perturbed so
that \(\lambda\) never occurs with multiplicity \(k\) or greater.
Taking \(k\) just larger than \(\sqrt d\), the resulting normalized
spectral defect is of order \(\sqrt d/r\), and hence tends to zero
exactly when \(d/r^2\) tends to zero.  The quadratic threshold in the
figure is thus imposed by the eigenvalue-multiplicity geometry of the
matrix fibers.

\begin{thm}\label{AHgamma}
    Let $A$ be a unital simple non-elementary AH algebra with subquadratic growth.  It follows that $A$ has uniform property $\Gamma$.
\end{thm}
\noindent
This theorem is optimal in light of the second author's construction of a unital simple AH algebra with quadratic dimension growth that fails to have uniform property $\Gamma$, and yields a clean picture: quadratic dimension growth is the precise threshold at which uniform property $\Gamma$ begins to fail.  

We defer the full definition of subquadratic growth for ASH algebras until Definition~\ref{def:rsh-subquadratic-growth}.  For now we note only that by a result of Ng and Winter (\cite[Corollary~2.1]{NgWinter}), every unital simple separable ASH algebra is an inductive limit of recursive subhomogeneous algebras in the sense of Phillips (\cite{PhillipsRSH}).  The RSH formulation of our main theorem is the following.

\begin{thm}
\label{ASHgamma}
Let \(A\) be a separable, unital \(C^*\)-algebra with no nonzero
finite-dimensional representations.  Suppose that \(A\) admits a
unital RSH inductive-limit decomposition with subquadratic RSH
dimension growth in the sense of
Definition~\ref{def:rsh-subquadratic-growth}.  Then \(A\) has uniform
property \(\Gamma\).
\end{thm}
\noindent
Note that for a simple unital \(C^*\)-algebra, the absence of nonzero
finite-dimensional representations is equivalent to
non-elementarity.  

In their foundational work on uniform property $\Gamma$, Castillejos et al.\ prove that the Toms--Winter conjecture holds among simple, unital, separable, nuclear, non-elementary C$^*$-algebras with uniform property $\Gamma$ (Theorem~5.6 of \cite{CETW:IMRN}). The following corollary of Theorem~\ref{ASHgamma} is then immediate:
\begin{cor}\label{TWASH}
The class of simple unital non-elementary inductive limits of recursive subhomogeneous C$^*$-algebras with subquadratic growth satisfies the Toms--Winter conjecture.
\end{cor}
\noindent
Theorem~\ref{ASHgamma} and Corollary~\ref{TWASH} in fact hold for the formally larger class of C$^*$-algebras that admit \emph{unital locally
tracially subquadratic RSH approximation} (see Definition~\ref{def:unital-local-rsh-subquadratic}).  It is also important to note that subquadratic ASH algebras as in Corollary~\ref{TWASH} regularly fail to satisfy the conjecturally equivalent conditions of the Toms--Winter conjecture---finite nuclear dimension, $\mathcal{Z}$-stability, and strict comparison.  In the subclass of subquadratic AH algebras there are a multitude of examples---see for instance \cite{Villadsen:JFA}, \cite{Villadsen:JAMS}, \cite{Toms:Ann}, \cite{Toms:JFA}---but the second author has also shown that these pathologies can persist in projectionless ASH algebras which, for K-theoretic reasons, are not AH \cite{Toms2005Independence}.

There is a philosophical consequence of Corollary~\ref{TWASH} worth noting as well.  Among unital simple separable nuclear C$^*$-algebras with a trace, we have no examples which are not ASH\@.  It is possible that none exist, although this is a deep and difficult problem.  Moreover, many naturally occurring C$^*$-algebras arising from groups and their dynamics are known, despite outward appearances, to be ASH through confirmations of Elliott's original classification conjecture for simple separable nuclear C$^*$-algebras.  In any case, confirming the Toms--Winter conjecture for ASH algebras would be a landmark result, and Corollary~\ref{TWASH} provides strong evidence that it will hold.  Indeed, any ASH counterexample to the conjecture would have to consist of an algebra that simultaneously enjoys strict comparison and quadratic or faster dimension growth, a pairing that all evidence suggests is unlikely in the extreme.  

We also prove a tracial local version of Theorem~\ref{AHgamma}.  In Definition~\ref{def:vaccaro-local-subquadratic} below we establish what it means for a unital C$^*$-algebra with a trace to have \emph{locally tracially subquadratic homogeneous approximation}, building on earlier definitions of Elliott--Niu and Vaccaro (\cite{ElliottNiu25}, \cite{Vaccaro:preprint}).  This notion permits approximating homogeneous subalgebras which are not necessarily unital, a vital piece of flexibility in applications to crossed products \cite[Lemma~4.3]{ElliottNiu25}.  
\begin{thm}\label{localgamma}
Let \(A\) be a simple, separable, unital, non-elementary
\(C^*\)-algebra with \(T(A)\neq\emptyset\).  If \(A\) has locally
tracially subquadratic homogeneous approximation, then \(A\) has
uniform property \(\Gamma\).
\end{thm}
\noindent
Again, the second author's construction of a unital simple AH algebra with quadratic dimension growth that fails uniform property $\Gamma$ shows that the subquadratic condition in Theorem~\ref{localgamma} is optimal.

The main technical tool behind these results is an iterated
rank-avoidance theorem for sections of self-adjoint endomorphism
bundles---here viewed as the self-adjoint elements of a homogeneous C$^*$-algebra---together with a relative form adapted to recursive
subhomogeneous algebras.  It converts the codimension calculation
described after Figure~\ref{fig:geometric-thresholds} above into a perturbation argument with uniform control of
spectral multiplicities.  The relevant bad spectral locus is generally a stratified space with smooth manifold strata, each of which is defined by an exact nullity value.  In particular, it cannot be avoided by a single application of smooth transversality.  In
Proposition~\ref{prop:compact-bundle-rank-avoidance}, we instead pass
iteratively through the strata.  Once the higher-nullity strata and the bad loci
associated to previously treated spectral values have been deleted,
the next stratum is a closed smooth subfibration of the remaining open
fiber subbundle.  A bundle-valued avoidance theorem of Jacob then removes
it without leaving that open subbundle, and hence without reintroducing
any stratum already avoided \cite{JacobDimAvoidance}.  This gives the spectral multiplicity estimates required
in the sequel.

A relative form of this method is required in the RSH
setting.  At a pullback block
\(S\oplus_{C(Y,M_r)}C(X,M_r)\), the section over \(Y\) is determined
by the preceding stage and already satisfies the required multiplicity
bounds.  Proposition~\ref{prop:relative-compact-rank-avoidance}
perturbs an extension over \(X\) while fixing the section on \(Y\).  Induction through the chosen pullback decomposition then
gives Theorem~\ref{thm:rsh-rank-avoidance}, without revisiting the
earlier blocks or accumulating error in proportion to the length of
the decomposition.

We close this section with a roadmap for the sequel.  Section~\ref{sec:ultrapower} introduces property \((FG)\), a finite-grid variation on Property \((S)\) from \cite{ElliottNiu25}, and establishes its connection to uniform property \(\Gamma\).  Section~\ref{sec:rank-avoidance} establishes the \(k^2\) codimension calculation and bundle-valued rank avoidance.  Section~\ref{sec:local} proves the unital local homogeneous theorem, while Section~\ref{sec:AH} deduces the AH theorem.  We note here that the AH theorem follows from its ASH counterpart, but we have deliberately addressed it separately as its proof is considerably less technical.
Section~\ref{sec:vaccaro-local} treats possibly nonunital homogeneous local models and proves the Vaccaro-style theorem. Section~\ref{sec:rsh-local} proves relative rank avoidance for arbitrary RSH decompositions and deduces the main RSH-inductive-limit theorem.

\vspace{2mm}
\noindent
{\bf Acknowledgements.}  The second author gratefully acknowledges the support of the Simons Foundation (SFI-MPS-TSM-00025606).

\vspace{2mm}
\noindent
{\bf AI Statement.} ChatGPT was used for language proofreading,
notational consistency, and reference checking. The mathematical
content of the paper is due solely to the authors.

\section{Uniform tracial ultrapowers and the property \texorpdfstring{\((FG)\)}{(S)} criterion}
\label{sec:ultrapower}

\subsection{Uniform tracial ultrapowers}

Throughout the paper all traces are tracial states, and we use \(T(A)\) to denote the set thereof.  If \(A\) is unital and \(T(A)\neq\emptyset\), set
\[
    \|a\|_{2,T(A)}:=\sup_{\tau\in T(A)}\tau(a^*a)^{1/2}.
\]
This is the \emph{uniform 2-norm relative to \(T(A)\)}.
We also use \(\|a\|_{2,T}\) for the corresponding seminorm over a nonempty subset \(T\subseteq T(A)\).

Fix a free ultrafilter \(\U\) on \(\N\) and let
\[
        c_{\U}(A)=\{(a_n)_n\in \ell^\infty(A): \lim_{\U}\|a_n\|_{2,T(A)}=0\}.
\]
The uniform tracial ultrapower of \(A\) is \(A^{\U}:=\ell^\infty(A)/c_{\U}(A)\). We write \((a_n)_{\U}\) for the class of a bounded sequence. When the algebra is clear from context, we write simply \(c_{\U}\) for \(c_{\U}(A)\).

We denote by
\[
        \iota_A\colon A\longrightarrow A^{\U}
\]
the unital \(*\)-homomorphism induced by constant sequences.  In the
nonsimple case this map need not be injective.  Accordingly, whenever
nonsimple algebras are under consideration, \(A^{\U}\cap A'\) means
\(A^{\U}\cap\iota_A(A)'\).

Every sequence \((\tau_n)_n\) in \(T(A)\) induces a trace on \(A^{\U}\) by
\[
        (a_n)_{\U}\longmapsto\lim_{\U}\tau_n(a_n).
\]
We call traces of this form \emph{limit traces}, and write \(T_{\U}(A)\)
for their weak-* closed convex hull in \(T(A^{\U})\).

\begin{defn}
\label{def:uniform-gamma-mcduff}
Let \(A\) be separable and unital, with \(T(A)\neq\emptyset\).
\begin{enumerate}[(1)]
\item The algebra \(A\) is \emph{uniformly McDuff} if, for every
      \(n\in\N\), there is a unital \(*\)-homomorphism
\[
        M_n\longrightarrow A^{\U}\cap\iota_A(A)'.
\]
\item The algebra \(A\) has \emph{uniform property \(\Gamma\)} if, for
      every \(n\in\N\), there are pairwise orthogonal projections
      \(p_1,\ldots,p_n\in A^{\U}\cap\iota_A(A)'\), summing to
      \(1_{A^{\U}}\), such that
\[
        \rho(\iota_A(a)p_i)=\frac1n\rho(\iota_A(a)),
        \qquad a\in A,\ \rho\in T_{\U}(A),\ 1\leq i\leq n.
\]
\end{enumerate}
These are the standard definitions from
\cite[Definitions~1.1 and~1.2]{Vaccaro:preprint}.
\end{defn}

For \(\delta>0\), fix once and for all a continuous function \(\eta_\delta\colon\R\to[0,1]\) such that \(\eta_\delta(t)=1\) for \(|t|\leq \delta/2\), \(\eta_\delta(t)=0\) for \(|t|\geq \delta\), and \(\eta_\delta\) is affine on \([\delta/2,\delta]\) and on \([-\delta,-\delta/2]\). By forming \(\eta_\delta(a-\lambda)\) we can thus roughly isolate the part of the spectrum of a self-adjoint element \(a\) lying near \(\lambda\).

\subsection{Property \((FG)\): a finite-grid property \((S)\) criterion}

Many of the proofs below will use a finite-grid form of property \((S)\) introduced in \cite{ElliottNiu25}, which we introduce presently.

\begin{defn}
\label{def:property-S}
Let \(A\) be a unital \(C^*\)-algebra with \(T(A)\neq\emptyset\). We say that \(A\) has \emph{property \((FG)\)} if for every sequence \(a_n\in A_{1,\mathrm{sa}}\) and every finite set \(\Lambda\subset (-1,1)\), there are self-adjoint contractions \(b_n\in A\) and numbers \(\delta_n>0\) such that \(\lim_{\U}\|a_n-b_n\|_{2,T(A)}=0\) and
\[
        \lim_{\U}\max_{\lambda\in\Lambda}\sup_{\tau\in T(A)}
        \tau\big(\eta_{\delta_n}(b_n-\lambda)\big)=0.
\]
\end{defn}

\begin{prop}
\label{prop:S-criterion}
Let \(A\) be a unital \(C^*\)-algebra with \(T(A)\neq\emptyset\). If \(A\) has property \((FG)\), then \(A^{\U}\) has real rank zero.
\end{prop}

\begin{proof}
Let \(q\colon \ell^\infty(A)\to A^{\U}\) be the quotient map. We first note that every self-adjoint contraction in \(A^{\U}\) has a representative by self-adjoint contractions in \(A\). Indeed, if \(x\in A^{\U}\) is self-adjoint and \(\|x\|\leq 1\), start with any lift and replace it by its self-adjoint part to obtain a self-adjoint lift \(d=(d_n)_n\in \ell^\infty(A)\). Let \(\chi\colon\R\to[-1,1]\) be the clipping function \(\chi(t)=\max\{-1,\min\{t,1\}\}\). Functional calculus commutes with quotients, so \(q((\chi(d_n))_n)=\chi(q(d))=\chi(x)=x\). Thus, after replacing \(d_n\) by \(\chi(d_n)\), we may lift \(x\) by a sequence \(a_n\in A_{1,\mathrm{sa}}\).

It is enough to approximate self-adjoint contractions in \(A^{\U}\) by self-adjoint elements with finite spectrum.  Let \(x=(a_n)_{\U}\in A^{\U}\) be a self-adjoint contraction with \(a_n\in A_{1,\mathrm{sa}}\) and fix \(\varepsilon>0\). Choose a partition \(-1=t_0<t_1<\cdots<t_m=1\) whose mesh is less than \(\varepsilon\) with \(m\geq2\). Put \(\Lambda=\{t_1,\ldots,t_{m-1}\}\), and let \(\mu_j=(t_{j-1}+t_j)/2\) for \(j=1,\ldots,m\).

By property \((FG)\), there are \(b_n\in A_{1,\mathrm{sa}}\) and \(\delta_n>0\) such that \((a_n-b_n)_n\in c_{\U}(A)\) and
\[
        \lim_{\U}\max_{1\leq j\leq m-1}\sup_{\tau\in T(A)}
        \tau\big(\eta_{\delta_n}(b_n-t_j)\big)=0.
\]
For each \(n\), choose \(\alpha_n>0\) such that \(\alpha_n\leq\delta_n/2\) and \(\alpha_n\leq (1/4)\min_{1\leq j\leq m}(t_j-t_{j-1})\). Define a continuous function \(h_n\colon[-1,1]\to[-1,1]\) by setting \(h_n=\mu_j\) on \([t_{j-1}+\alpha_n,t_j-\alpha_n]\) for \(2\leq j\leq m-1\), on \([t_0,t_1-\alpha_n]\) for \(j=1\), and on \([t_{m-1}+\alpha_n,t_m]\) for \(j=m\), and on each interval \([t_j-\alpha_n,t_j+\alpha_n]\), \(1\leq j\leq m-1\), let \(h_n\) be affine from \(\mu_j\) to \(\mu_{j+1}\). Set
\[
        m_0=\max_{1\leq j\leq m}(t_j-t_{j-1}).
\]
Our choice of \(\alpha_n\) then gives
\[
        \sup_{t\in[-1,1]}|h_n(t)-t|
        \leq 3m_0/4<\varepsilon
\]
for every \(n\). With \(y=(h_n(b_n))_{\U}\), we have \(x=(b_n)_{\U}\), and the quotient norm is bounded by the \(\ell^\infty\)-norm of a representative. Hence
\[
\begin{aligned}
        \|x-y\|
        &\leq
        \big\|(b_n-h_n(b_n))_n\big\|_{\ell^\infty(A)}\\
        &=\sup_n\|b_n-h_n(b_n)\|\\
        &\leq 3m_0/4
        <\varepsilon.
\end{aligned}
\]

It remains to prove that \(y\) has finite spectrum. Let \(p(s)=\prod_{j=1}^m(s-\mu_j)\) for \(s \in [-1,1]\). The function \(p\circ h_n\) vanishes off the union of the intervals \([t_j-\alpha_n,t_j+\alpha_n]\), \(1\leq j\leq m-1\), because outside those transition intervals the function \(h_n\) is one of the constants \(\mu_1,\ldots,\mu_m\). Set 
\[
M=\sup_{s\in[-1,1]}|p(s)|^2
\]
and fix \(t\in[-1,1]\). If \(t\) lies outside the union
of the transition intervals, then \(p(h_n(t))=0\). Otherwise,
\(t\in[t_j-\alpha_n,t_j+\alpha_n]\) for some \(1\leq j\leq m-1\), so
\[
        |t-t_j|\leq\alpha_n\leq\frac{\delta_n}{2}.
\]
It follows that \(\eta_{\delta_n}(t-t_j)=1\). Moreover,
\(h_n(t)\in[-1,1]\), and hence \(|p(h_n(t))|^2\leq M\). Thus, in either
case,
\[
        |p(h_n(t))|^2
        \leq M\sum_{j=1}^{m-1}\eta_{\delta_n}(t-t_j),
        \ t\in[-1,1].
\]
Since \(b_n\) is a self-adjoint contraction, we have
\(\sigma(b_n)\subseteq[-1,1]\). Functional calculus and the inequality above therefore yield
\[
        p(h_n(b_n))^*p(h_n(b_n))
        \leq
        M\sum_{j=1}^{m-1}\eta_{\delta_n}(b_n-t_j).
\]
Applying an arbitrary \(\tau\in T(A)\) and taking suprema we have
\[
\begin{aligned}
        \|p(h_n(b_n))\|_{2,T(A)}^2
        &=\sup_{\tau\in T(A)}
          \tau\big(p(h_n(b_n))^*p(h_n(b_n))\big)\\
        &\leq
          M\sup_{\tau\in T(A)}
          \sum_{j=1}^{m-1}
          \tau\big(\eta_{\delta_n}(b_n-t_j)\big)\\
        &\leq
          M\sum_{j=1}^{m-1}\sup_{\tau\in T(A)}
          \tau\big(\eta_{\delta_n}(b_n-t_j)\big).
\end{aligned}
\]
The right-hand side tends to zero along \(\U\). Thus
\((p(h_n(b_n)))_n\in c_{\U}(A)\), so \(p(y)=0\) in \(A^{\U}\).
Since \(h_n\) is real-valued and \(b_n\) is self-adjoint, each
\(h_n(b_n)\) is self-adjoint, and hence so is \(y\).  Now
the polynomial spectral mapping theorem gives
\(\sigma(y) \subseteq p^{-1}(0) = \{\mu_1,\ldots,\mu_m\}\). It follows that
 \(y\) is a
self-adjoint finite-spectrum element within \(\varepsilon\) of \(x\).
Finite-spectrum self-adjoint elements are therefore dense in
\((A^{\U})_{\mathrm{sa}}\), and \(A^{\U}\) has real rank zero.
\end{proof}

\subsection{The bridge to uniform property \texorpdfstring{\(\Gamma\)}{Gamma}}

In the sequel we will deduce uniform property \(\Gamma\) from the real rank zero property for \(A^{\U}\) through recent work of Vaccaro.

\begin{defn} \cite[Definition 1.5]{Vaccaro:preprint}
\label{def:tlfnd}
A unital \(C^*\)-algebra \(A\) with \(T(A)\neq\emptyset\) has
\emph{tracially locally finite nuclear dimension} if, for every finite
subset \(\mathcal F\subset A_1\) and every \(\varepsilon>0\), there is a
\(C^*\)-subalgebra \(C\subseteq A\) with finite nuclear dimension such
that, for every \(a\in\mathcal F\), there is \(c\in C_1\) with
\[
        \|a-c\|_{2,T(A)}<\varepsilon.
\]
\end{defn}

The simple case of the bridge below is
\cite[Theorem~2.3]{Vaccaro:preprint}.  We use tracial almost divisibility
in the sense of \cite[Definition~1.3]{Vaccaro:preprint}.  We first record
that the relevant implication in Vaccaro's argument does not require
simplicity.

\begin{lem}
\label{lem:winter-centralization}
Let \(A\) be a separable unital \(C^*\)-algebra with
\(T(A)\neq\emptyset\).  If \(A\) is tracially almost divisible and has
tracially locally finite nuclear dimension, then \(A\) is uniformly
McDuff.
\end{lem}

\begin{proof}
Fix \(n\in\N\). Choose increasing finite sets
\(\mathcal F_j\subset A_1\) with norm-dense union and numbers
\(\varepsilon_j\searrow0\). By tracially locally finite nuclear
dimension, for each \(j\) there are a finite-nuclear-dimension
subalgebra \(C_j\subseteq A\) and elements
\(c_{j,a}\in(C_j)_1\), \(a\in\mathcal F_j\), such that
\[
        \|a-c_{j,a}\|_{2,T(A)}<\varepsilon_j.
\]
Put \(\widetilde C_j=C^*(C_j,1_A)\). The canonical map from the
minimal unitization of \(C_j\) onto \(\widetilde C_j\) is surjective.
Hence \(\widetilde C_j\) has finite nuclear dimension by
\cite[Proposition~2.3(iv) and Remark~2.11]{WinterZacharias:AIM}.

We use the finite-set conclusion of
\cite[Lemma~5.11]{Winter:IM12}. Although that lemma is stated for a
simple ambient algebra, inspection of its proof, and of the
intermediate propositions used there, shows that simplicity and
non-elementarity are not used once tracial almost divisibility is
assumed. The construction uses only separability and unitality of the
ambient algebra, tracial almost divisibility, finite nuclear dimension
of the subalgebra, functional calculus, trace estimates, and diagonal
sequence arguments. In particular, no step passes from a nonzero
positive element to a full one or invokes faithfulness of traces.
Thus the same proof, applied to \(\widetilde C_j\), gives a c.p.c.\
order zero map
\(\varphi_j\colon M_n\longrightarrow A\)
such that
\[
        \|[\varphi_j(x),c_{j,a}]\|
        <\varepsilon_j\|x\|,
        \qquad
        a\in\mathcal F_j,\ x\in M_n,
\]
and
\[
        \tau(\varphi_j(1_n))>1-\varepsilon_j,
        \qquad \tau\in T(A).
\]
This is the finite-set fact used in
\cite[Proposition~1.6, \((2)\Rightarrow(1)\)]{Vaccaro:preprint}.

For \(a\in\mathcal F_j\) and \(x\in M_n\),
\[
\begin{aligned}
        \|[\varphi_j(x),a]\|_{2,T(A)}
        &\leq
        \|[\varphi_j(x),c_{j,a}]\|
        +2\|x\|\,\|a-c_{j,a}\|_{2,T(A)} \\
        &<3\varepsilon_j\|x\|.
\end{aligned}
\]
Since the sets \(\mathcal F_j\) increase to a norm-dense subset of
\(A_1\) and the maps \(\varphi_j\) are contractive, and
\[
        \lim_{j\to\infty}
        \|[\varphi_j(x),a]\|_{2,T(A)}=0,
        \qquad a\in A,\ x\in M_n.
\]
Thus
\[
        \Phi\colon M_n\longrightarrow A^{\U},
        \qquad
        \Phi(x)=(\varphi_j(x))_{\U},
\]
has range in \(A^{\U}\cap\iota_A(A)'\).

Moreover, \(0\leq\varphi_j(1_n)\leq1_A\), and therefore
\[
\begin{aligned}
        \|1_A-\varphi_j(1_n)\|_{2,T(A)}^2
        &\leq
        \sup_{\tau\in T(A)}
        \tau(1_A-\varphi_j(1_n)) \\
        &<\varepsilon_j.
\end{aligned}
\]
Hence \(\Phi(1_n)=1_{A^{\U}}\). A unital c.p.c.\ order zero map from
\(M_n\) is a \(*\)-homomorphism. Since \(n\) was arbitrary, \(A\) is
uniformly McDuff.
\end{proof}

In the nonsimple setting, the absence of finite-dimensional
representations supplies the small full elements needed in Vaccaro's
real-rank-zero argument.

\begin{thm}
\label{thm:VW-bridge}
Let \(A\) be a separable, unital \(C^*\)-algebra with
\(T(A)\neq\emptyset\) and no nonzero finite-dimensional
representations.  Suppose that \(A\) has tracially locally finite
nuclear dimension.  If \(A^{\U}\) has real rank zero, then \(A\) is
uniformly McDuff, and hence has uniform property \(\Gamma\).
\end{thm}

\begin{proof}
Fix \(m\in\N\), set \(D=M_m(A)\), and write
\[
        B=D^{\U}\cong M_m(A^{\U}).
\]
Then \(B\) has real rank zero by
\cite[Theorem~2.10]{BrownPedersen:JFA}. It has no nonzero
finite-dimensional representations: a nonzero finite-dimensional
representation of \(B\), restricted along the unital constant-sequence
homomorphism \(D\to B\), would give a nonzero finite-dimensional
representation of \(D\), and hence, by restriction to a full corner,
of \(A\). Also \(T(B)\neq\emptyset\), since every trace on \(D\)
induces a limit trace on \(B\).

Let
\[
        P=\bigotimes_{j=1}^{\infty}(M_2\oplus M_3).
\]
By \cite[Corollary~7]{ElliottRordam:Abel}, there is a unital embedding
\(\psi\colon P\longrightarrow B\).
For each \(N\), choose in the first \(N\) tensor factors a projection
\(p_N\) having rank one in every simple summand, and tensor it with
the unit in the tail. The projection \(p_N\) is full in \(P\), so
\(q_N=\psi(p_N)\) is full in \(B\): the ideal generated by \(q_N\)
contains \(\psi(P)\), and hence contains \(1_B\). Every trace on \(B\)
restricts to a trace on \(P\), while every simple summand of the first
\(N\) tensor factors has size at least \(2^N\). Thus
\[
        \sup_{\rho\in T(B)}d_\rho(q_N)
        =
        \sup_{\rho\in T(B)}\rho(q_N)
        \leq2^{-N}.
\]

We now prove tracial almost divisibility. Fix
\(a\in D_{1,+}\), \(n\in\N\), and \(\varepsilon>0\), and let
\(\bar a\in B\) be the constant-sequence image of \(a\). For every
\(\delta>0\), choose \(N\) with \(2^{-N}<\delta\). Since \(q_N\) is
full,
\[
        \bar a\in\overline{Bq_NB},
        \qquad
        d_\rho(q_N)<\delta,
        \quad \rho\in T(B).
\]
Therefore \cite[Lemma~2.1]{Vaccaro:preprint}, applied in \(B\), gives
a \(*\)-homomorphism
\(\Phi\colon M_n\longrightarrow\overline{\bar aB\bar a}\)
such that
\[
        \rho(\Phi(1_n))
        \geq\rho(\bar a)-\varepsilon/2,
        \qquad \rho\in T(B).
\]

Let
\(Q\colon\ell^\infty(D)\longrightarrow B\)
be the quotient map. Since a surjective \(*\)-homomorphism maps the
hereditary subalgebra generated by a positive element onto the
hereditary subalgebra generated by its image, the restriction
\[
        \overline{a\ell^\infty(D)a}
        \longrightarrow
        \overline{\bar aB\bar a}
\]
is surjective. By the order-zero lifting argument used in
\cite[Proposition~1.6, \((1)\Rightarrow(2)\)]{Vaccaro:preprint},
\(\Phi\) lifts to a sequence of c.p.c.\ order zero maps
\[
        \varphi_j\colon
        M_n\longrightarrow\overline{aDa}
\]
such that
\[
        (\varphi_j(x))_{\U}=\Phi(x),
        \qquad x\in M_n.
\]

We claim that
\[
\begin{split}
        G=\bigl\{j\in\N:
        \tau(\varphi_j(1_n))>\tau(a)-\varepsilon
        \text{ for every }\tau\in T(D)\bigr\}
        \in\U.
\end{split}
\]
Otherwise, for every \(j\) in the \(\U\)-large complement of \(G\),
choose \(\tau_j\in T(D)\) such that
\[
        \tau_j(\varphi_j(1_n))
        \leq\tau_j(a)-\varepsilon,
\]
and choose \(\tau_j\) arbitrarily at the remaining indices. The
associated limit trace \(\rho\in T(B)\) would satisfy
\[
        \rho(\Phi(1_n))
        \leq\rho(\bar a)-\varepsilon,
\]
contrary to the choice of \(\Phi\). Thus any \(j\in G\) gives the
order zero map required for tracial almost divisibility of \(a\) in
\(D\). Since \(a,n,\varepsilon\), and \(m\) were arbitrary, \(A\) is
tracially almost divisible.

Lemma~\ref{lem:winter-centralization} now shows that \(A\) is uniformly
McDuff. To see uniform property \(\Gamma\) directly, fix \(n\) and a
unital \(*\)-homomorphism
\[
        \Theta\colon M_n\longrightarrow
        A^{\U}\cap\iota_A(A)'.
\]
Put
\[
        p_i=\Theta(e_{ii}),
        \qquad
        v_{ij}=\Theta(e_{ij}).
\]
For every \(a\in A\), every \(\rho\in T_{\U}(A)\), and every \(i,j\),
\[
\begin{aligned}
        \rho(\iota_A(a)p_i)
        &=\rho(\iota_A(a)v_{ij}v_{ij}^*) \\
        &=\rho(v_{ij}^*\iota_A(a)v_{ij}) \\
        &=\rho(\iota_A(a)v_{ij}^*v_{ij}) \\
        &=\rho(\iota_A(a)p_j).
\end{aligned}
\]
Since \(\sum_i p_i=1\), each of these values is
\[
        \frac1n\rho(\iota_A(a)).
\]
Hence \(A\) has uniform property \(\Gamma\).
\end{proof}

\section{Rank avoidance below quadratic scale}
\label{sec:rank-avoidance}

\subsection{The self-adjoint rank-defect locus}

The source of the quadratic phenomenon that governs our results is an elementary dimension count.

\begin{lem}
\label{lem:rank-defect-codim}
Let \(r\in\N\), \(0\leq j\leq r\), and \(\lambda\in\R\).  The exact-nullity locus
\[
        \Sigma^{=j}_{r,\lambda}
        =\{a\in M_r(\C)_{\mathrm{sa}}:\dim\ker(a-\lambda1_r)=j\}
\]
is a smooth, possibly disconnected, embedded submanifold of
\(M_r(\C)_{\mathrm{sa}}\) of real codimension \(j^2\).  It follows that
for \(1\leq k\leq r\),
\[
        \Sigma_{r,k,\lambda}
        =\{a\in M_r(\C)_{\mathrm{sa}}:\dim\ker(a-\lambda1_r)\geq k\}
        =\bigcup_{j=k}^r\Sigma^{=j}_{r,\lambda}
\]
is a closed stratified subset whose largest stratum has real dimension
\(r^2-k^2\), and hence has real codimension \(k^2\) in the stratified
sense.
\end{lem}

\begin{proof}
Translation by \(\lambda1_r\) reduces the assertion to \(\lambda=0\).
An element of \(\Sigma^{=j}_{r,0}\) is determined by its kernel
\(K\in\operatorname{Gr}(j,r)\) and an invertible self-adjoint operator
on \(K^\perp\).  Thus \(\Sigma^{=j}_{r,0}\) is the total space of the
open bundle of invertible self-adjoint endomorphisms of the orthogonal
complement bundle over \(\operatorname{Gr}(j,r)\).  It is therefore a
smooth embedded submanifold.  Its real dimension is
\[
        2j(r-j)+(r-j)^2=r^2-j^2,
\]
so its real codimension is \(j^2\).  The final assertions follow by
taking the union over \(j\geq k\); the stratum of largest dimension is
the one with \(j=k\).
\end{proof}

\subsection{Bundle-valued transversality}

We use a transversality argument below applied to a real vector bundle,
even though the bundle \(E\) appearing in the proposition is complex.
This is because a rank-\(r\) complex vector bundle \(E\to X\) equipped
with a continuous Hermitian metric has the property that its self-adjoint endomorphisms form a locally trivial real vector bundle with fiber \(M_r(\C)_{\mathrm{sa}}\) (a real vector
space of dimension \(r^2\)). Thus a self-adjoint element of
\(\Gamma(\operatorname{End}(E))\) is exactly a continuous section of
\(\operatorname{End}(E)_{\mathrm{sa}}\).

For \(\lambda\in\R\) and \(1\leq j\leq r\), let
\[
        D_{\lambda,j}(x)
        =
        \{T\in\operatorname{End}(E_x)_{\mathrm{sa}}:
        \dim\ker(T-\lambda1_{E_x})\geq j\},
\]
and put \(D_{\lambda,r+1}(x)=\emptyset\). Unitary conjugation preserves
these loci, so they form closed fiberwise subsets of
\(\operatorname{End}(E)_{\mathrm{sa}}\). The exact-nullity locus
\[
        S_{\lambda,j}(x)
        =
        D_{\lambda,j}(x)\setminus D_{\lambda,j+1}(x)
\]
is a smooth embedded submanifold of the fiber. After
\(D_{\lambda,j+1}\) has been deleted, these loci form a closed locally
trivial subfibration of the resulting open locally trivial
subfibration. Its fiber has real codimension \(j^2\).

A smooth embedded submanifold of codimension \(j^2\) has Lipschitz
codimension \(j^2\) in the sense of Jacob. Consequently, when
\(j^2>d\), every open subset \(U\subseteq X\) is normal and satisfies
\(\dim U\leq d<j^2\), so
\cite[Theorem~3.4]{JacobDimAvoidance} makes the exact-nullity locus
\(U\)-avoidable. The bundle-valued conclusion, including its relative
form over a closed subset, is then supplied by
\cite[Theorem~5.1]{JacobDimAvoidance}. We apply these results one
exact-nullity stratum at a time.

\begin{prop}
\label{prop:compact-bundle-rank-avoidance}
Let \(X\) be a compact metrizable space with finite covering dimension \(d\), and let \(E\to X\) be a complex vector bundle of constant rank \(r\). Let \(\Lambda\subset(-1,1)\) be finite. If \(1\leq k\leq r\) and \(d<k^2\), then for every \(f\in\Gamma(\operatorname{End}(E))_{1,\mathrm{sa}}\) and every \(\varepsilon>0\), there exists \(g\in\Gamma(\operatorname{End}(E))_{1,\mathrm{sa}}\) such that \(\|f-g\|<\varepsilon\) and \(\dim\ker(g(x)-\lambda 1_{E_x})<k\) for every \(x\in X\) and every \(\lambda\in\Lambda\).
\end{prop}

\begin{proof}
If \(\Lambda=\emptyset\), take \(g=f\). Otherwise, enumerate
\(\Lambda=\{\lambda_1,\ldots,\lambda_L\}\),
let \(D_{\lambda,j}\) and \(S_{\lambda,j}\) denote the fiberwise loci
introduced above, and put
\(N=L(r-k+1)\).
We treat the pairs
\[
        (\ell,j),
        \qquad
        1\leq\ell\leq L,\quad r\geq j\geq k,
\]
first by increasing \(\ell\), and, for each fixed \(\ell\), by
decreasing \(j\).

Suppose all earlier strata have been avoided and the current pair is
\((\ell,j)\). The current section takes values in the open locally
trivial subfibration whose fiber is
\[
        V_{\ell,j}
        =
        M_r(\C)_{\mathrm{sa}}
        \setminus
        \left(
        \bigcup_{h<\ell}D_{\lambda_h,k}
        \cup D_{\lambda_\ell,j+1}
        \right).
\]
Inside \(V_{\ell,j}\), the set
\(S_{\lambda_\ell,j}\cap V_{\ell,j}\)
is closed, and the corresponding fiberwise sets form a closed locally
trivial subfibration. Its fiber is a smooth embedded submanifold of
real, and hence Lipschitz, codimension
\[
        j^2\geq k^2>d
\]
by Lemma~\ref{lem:rank-defect-codim}.

Every open subset \(U\subseteq X\) is metrizable and normal, with
\(\dim U\leq d\). Thus
\cite[Theorem~3.4]{JacobDimAvoidance} shows that the current bad
subfibration is \(U\)-avoidable for every such \(U\). Since \(X\) is
compact metrizable, it is paracompact, and
\cite[Theorem~5.1]{JacobDimAvoidance}, applied with tolerance
\(\varepsilon/(2N)\), perturbs the section within the open
subfibration so as to avoid the current stratum. Because the
perturbation remains in \(V_{\ell,j}\), all previously obtained
avoidance properties are preserved.

After the \(N\) steps, we obtain a self-adjoint section \(g_0\) such
that
\(\|f-g_0\|<\varepsilon/2\)
and
\[
        \dim\ker(g_0(x)-\lambda1_{E_x})<k,
        \qquad
        x\in X,\ \lambda\in\Lambda.
\]
Let
\[
        \chi(t)=\max\{-1,\min\{t,1\}\},
        \qquad t\in\R,
\]
and put \(g=\chi(g_0)\). Since \(f\) is a contraction and
\(\|f-g_0\|<\varepsilon/2\),
\[
        \|g_0-\chi(g_0)\|<\varepsilon/2,
\]
so \(\|f-g\|<\varepsilon\). Finally, for
\(\lambda\in(-1,1)\), one has \(\chi(t)=\lambda\) if and only if
\(t=\lambda\). Hence
\[
        \ker(g(x)-\lambda1_{E_x})
        =
        \ker(g_0(x)-\lambda1_{E_x}),
\]
and the kernel estimates are unchanged.
\end{proof}

\begin{rmk}
The restriction \(\Lambda\subset(-1,1)\) is harmless for the applications to real rank zero, where \(\Lambda\) is a finite set of interior cut points for a partition of \([-1,1]\). It allows the final clipping step in the proof to preserve the avoided spectral levels.
\end{rmk}

\begin{cor}
\label{cor:compact-tracial-rank-avoidance}
Let \(X\), \(E\), \(r\), \(d\), \(k\), and \(\Lambda\) be as in Proposition~\ref{prop:compact-bundle-rank-avoidance}, and put \(B=\Gamma(\operatorname{End}(E))\). Let \(T\subseteq T(B)\) be compact. If \(d<k^2\), then for every \(f\in B_{1,\mathrm{sa}}\) and every \(\varepsilon>0\), there are \(g\in B_{1,\mathrm{sa}}\) and \(\delta>0\) such that \(\|f-g\|<\varepsilon\) and
\[
        \tau\big(\eta_\delta(g-\lambda)\big)<\frac{k}{r},
        \qquad \tau\in T,\ \lambda\in\Lambda.
\]
\end{cor}

\begin{proof}
Apply Proposition~\ref{prop:compact-bundle-rank-avoidance} to obtain
\(g\) with \(\|f-g\|<\varepsilon\) and
\(\dim\ker(g(x)-\lambda1_{E_x})<k\) for all \(x\in X\) and
\(\lambda\in\Lambda\).  If \(\Lambda=\emptyset\), any \(\delta>0\)
finishes the proof.  For \(\lambda\in\Lambda\), let
\(\alpha_\lambda(x)\) be the \(k\)-th eigenvalue of
\(|g(x)-\lambda1_{E_x}|\)
listed in nondecreasing order.  This number is strictly positive.
Moreover, \(\alpha_\lambda\) is continuous.  Indeed, over a unitary
trivialization of \(E\), the displayed field is norm-continuous in
\(M_r(\C)_+\), and the min--max principle gives
\[
        |s_j(A)-s_j(B)|\leq\|A-B\|,
        \qquad A,B\in M_r(\C)_+,\ 1\leq j\leq r,
\]
for the ordered eigenvalues.  The resulting local functions agree on
overlaps because ordered eigenvalues are invariant under unitary
conjugation.

For \(\tau\in T(B)\), the restriction of \(\tau\) to the central copy
of \(C(X)\) is integration against a Borel probability measure
\(\mu_\tau\). Since each fiber \(M_r\) has a unique normalized trace,
local trivializations and a partition of unity give
\[
        \tau(h)
        =
        \int_X\tr_r(h(x))\,d\mu_\tau(x),
        \qquad h\in B.
\]
The preceding pointwise estimate therefore gives
\[
        \tau\big(\eta_\delta(g-\lambda)\big)
        \leq\frac{k-1}{r}
        <\frac{k}{r},
        \qquad
        \tau\in T(B),\ \lambda\in\Lambda.
\]
In particular, the estimate holds for every \(\tau\in T\).
\end{proof}

\section{Local tracial subquadratic approximation}
\label{sec:local}

\subsection{Real rank zero of the uniform tracial ultrapower}
Here we relax Vaccaro's notion of locally flat dimension growth to a more natural subquadratic condition, and prove that it suffices for uniform property \(\Gamma\) via real rank zero for the tracial ultrapower.

\begin{defn}
\label{def:unital-local-subquadratic}
Let \(A\) be a unital \(C^*\)-algebra with \(T(A)\neq\emptyset\). We say that \(A\) has \emph{unital locally tracially subquadratic homogeneous approximation} if, for every finite subset \(\mathcal F\subset A_1\), every \(\varepsilon>0\), every \(\eta>0\), and every \(R\in\N\), there is a unital subalgebra \(C\subseteq A\), with \(1_C=1_A\), of the form
\[
        C\cong\bigoplus_{i=1}^m p_iM_{N_i}(C(X_i))p_i,
\]
where each \(X_i\) is compact metrizable of finite covering dimension and each \(p_i\) has constant rank \(r_i\), such that every element of \(\mathcal F\) is within \(\varepsilon\) in \(\|\cdot\|_{2,T(A)}\) of an element of \(C_1\), while \(r_i\geq R\) for all \(i\) and \(\max_i\dim(X_i)/r_i^2<\eta\).
\end{defn}

\begin{thm}
\label{thm:local-rr0}
Let \(A\) be a unital \(C^*\)-algebra with \(T(A)\neq\emptyset\). If \(A\) has unital locally tracially subquadratic homogeneous approximation, then \(A^{\U}\) has real rank zero.
\end{thm}

\begin{proof}
We verify property \((FG)\) from Definition~\ref{def:property-S} and then apply Proposition~\ref{prop:S-criterion}. Let \(a_n\in A_{1,\mathrm{sa}}\) be any sequence, and let \(\Lambda\subset(-1,1)\) be finite.

Choose numbers \(\theta_n\in(0,1/2)\) with \(\theta_n\to0\). For each \(n\), put \(\eta_n=\theta_n^2/4\), and choose \(R_n\in\N\) with \(1/R_n<\theta_n/4\). Apply Definition~\ref{def:unital-local-subquadratic} to \(\{a_n\}\) with \(\varepsilon = 1/n\), \(\eta = \eta_n\), and \(R = R_n\). This yields a unital subalgebra \(C_n\subseteq A\) of the form
\[
        C_n=\bigoplus_{i=1}^{m_n}B_{n,i},
        \qquad
        B_{n,i}=p_{n,i}M_{N_{n,i}}(C(X_{n,i}))p_{n,i},
\]
together with \(d_n\in(C_n)_1\) such that \( \|a_n-d_n\|_{2,T(A)}<1/n \).
Put \(c_n=(d_n+d_n^*)/2\).
Then \(c_n\in(C_n)_{1,\mathrm{sa}}\) and
\[
        \|a_n-c_n\|_{2,T(A)}
        \leq \|a_n-d_n\|_{2,T(A)}<1/n.
\]
Write \(c_n=c_{n,1}\oplus\cdots\oplus c_{n,m_n}\),
and let \(r_{n,i}=\rank(p_{n,i})\). We have
\[
        r_{n,i}\geq R_n
        \quad\text{and}\quad
        \frac{\dim(X_{n,i})}{r_{n,i}^2}<\eta_n.
\]

Set \(k_{n,i}=\lceil \theta_n r_{n,i}/2\rceil\). Then \(k_{n,i}/r_{n,i}<\theta_n\) because \(1/r_{n,i}<\theta_n/4\), and \(\dim(X_{n,i})<k_{n,i}^2\). Let \(e_{n,i}\) be the unit of \(B_{n,i}\) and define
\[
        T_{n,i}=\overline{\left\{\frac{\tau|_{B_{n,i}}}{\tau(e_{n,i})}:\tau\in T(A),\ \tau(e_{n,i})>0\right\}}\subseteq T(B_{n,i}).
\]
The set \(T_{n,i}\) is compact, possibly empty, because it is a closed
subset of the compact space \(T(B_{n,i})\). If \(T_{n,i}=\emptyset\),
the trace estimate below is vacuous. Applying
Corollary~\ref{cor:compact-tracial-rank-avoidance} to \(B_{n,i}\), the
compact trace set \(T_{n,i}\), the element \(c_{n,i}\), the set
\(\Lambda\), and the integer \(k_{n,i}\), we obtain \(b_{n,i}\in(B_{n,i})_{1,\mathrm{sa}}\) and \(\delta_{n,i}>0\) such that \(\|b_{n,i}-c_{n,i}\|<1/n\) and
\[
        \rho\big(\eta_{\delta_{n,i}}(b_{n,i}-\lambda)\big)<\frac{k_{n,i}}{r_{n,i}}<\theta_n,
        \qquad \rho\in T_{n,i},\ \lambda\in\Lambda.
\]
Put \(b_n=b_{n,1}\oplus\cdots\oplus b_{n,m_n}\) and \(\delta_n=\min_i\delta_{n,i}\). Then \(\|a_n-b_n\|_{2,T(A)}<2/n\), so \((a_n-b_n)_n\in c_{\U}\).

Fix \(\lambda\in\Lambda\) and \(\tau\in T(A)\). If \(\tau(e_{n,i})>0\), then the normalized restriction \(\tau|_{B_{n,i}}/\tau(e_{n,i})\) belongs to \(T_{n,i}\). If instead \(\tau(e_{n,i})=0\), then the corresponding summand contributes nothing. Since \(\eta_{\delta_n}\leq \eta_{\delta_{n,i}}\), we get
\[
        \tau\big(\eta_{\delta_n}(b_n-\lambda)\big)
        \leq\sum_{i=1}^{m_n}\tau(e_{n,i})\theta_n=\theta_n.
\]
Thus \(\max_{\lambda\in\Lambda}\sup_{\tau\in T(A)}\tau(\eta_{\delta_n}(b_n-\lambda))\leq\theta_n\), which tends to zero along \(\U\). This is property \((FG)\), so Proposition~\ref{prop:S-criterion} gives \(\rr(A^{\U})=0\).
\end{proof}

\subsection{Uniform property \texorpdfstring{\(\Gamma\)}{Gamma}}

\begin{lem}
\label{lem:local-tlfnd}
If \(A\) has unital locally tracially subquadratic homogeneous approximation, then \(A\) has tracially locally finite nuclear dimension.
\end{lem}

\begin{proof}
It is enough to consider finite sets of contractions. Given such a set
\(\mathcal F\) and \(\varepsilon>0\), apply
Definition~\ref{def:unital-local-subquadratic} with any positive
dimension--rank tolerance and with rank bound \(R=1\). The resulting
algebra is a finite direct sum of homogeneous algebras
\[
        p_iM_{N_i}(C(X_i))p_i
\]
over finite-dimensional compact spaces.

By \cite[Propositions~2.3--2.5]{WinterZacharias:AIM},
\(C(X_i)\) has finite nuclear dimension, and finite nuclear dimension
is preserved under matrix amplification, passage to hereditary
subalgebras, and finite direct sums. Hence the resulting homogeneous
subalgebra has finite nuclear dimension and approximates
\(\mathcal F\) in \(\|\cdot\|_{2,T(A)}\), as required.
\end{proof}

\begin{thm}
\label{thm:local-gamma}
Let \(A\) be separable, simple, unital, and non-elementary, with
\(T(A)\neq\emptyset\). If \(A\) has unital locally tracially
subquadratic homogeneous approximation, then \(A\) has uniform
property \(\Gamma\).
\end{thm}

\begin{proof}
Since \(A\) is simple and non-elementary, it has no nonzero
finite-dimensional representations.  By Theorem~\ref{thm:local-rr0},
the uniform tracial ultrapower \(A^{\U}\) has real rank zero, and by
Lemma~\ref{lem:local-tlfnd}, \(A\) has tracially locally finite nuclear
dimension.  Theorem~\ref{thm:VW-bridge} then gives uniform property
\(\Gamma\).
\end{proof}

\section{The unital AH case}
\label{sec:AH}

\subsection{Strictly subquadratic AH decompositions}

For this section, an AH decomposition means an inductive system \((A_n,\phi_n)\) with unital injective connecting maps and
\[
        A_n=\bigoplus_{i=1}^{m_n}p_{n,i}M_{N_{n,i}}(C(X_{n,i}))p_{n,i},
\]
where each \(X_{n,i}\) is compact metrizable of finite covering dimension and each \(p_{n,i}\) has constant rank. The injectivity hypothesis is harmless since replacing building blocks by their images in the limit only reduces the covering dimension of the underlying spaces.

\begin{defn}
\label{def:AH-subquadratic}
For a unital AH decomposition as above, set
\[
        r_n=\min_i\rank(p_{n,i})
        \quad\text{and}\quad
        q_n=\max_i\frac{\dim(X_{n,i})}{\rank(p_{n,i})^2}.
\]
The decomposition has \emph{subquadratic growth} if
\[
        \liminf_{n\to\infty}q_n=0,
\]
and it has \emph{strictly subquadratic growth} if
\(r_n\rightarrow\infty\) and
\(q_n\rightarrow0\).
A unital AH algebra has subquadratic growth if it admits a unital AH decomposition with subquadratic growth.
\end{defn}

The next result shows that we may as well always work with strictly subquadratic systems and the structure of the proof is probably well known to experts.

\begin{prop}
\label{prop:AH-subquadratic-to-strict}
Let \(A=\varinjlim(A_n,\phi_n)\) be a simple, unital, non-elementary AH algebra, and suppose that the displayed AH decomposition has subquadratic growth in the sense of Definition~\ref{def:AH-subquadratic}.  Then some telescoping of this decomposition has strictly subquadratic growth.
\end{prop}

\begin{proof}
We first show that \(r_n\to\infty\).  Every irreducible representation of \(A_n\) has dimension \(\rank(p_{n,i})\) for some \(i\), so \(r_n\) is the least dimension of an irreducible representation of \(A_n\).  If \(\pi\) is an irreducible representation of \(A_{n+1}\), then \(\pi\circ\phi_n\) is a unital finite-dimensional representation of \(A_n\).  Decomposing it into irreducible representations gives \(\dim(\pi)\geq r_n\).
Taking the minimum over the irreducible representations of \(A_{n+1}\) shows that \(r_{n+1}\geq r_n\).

Suppose, toward a contradiction, that \((r_n)\) is bounded above by \(R\) and let \(D=R!\).  Let \(\operatorname{Rep}_D(A_n)\)  be the space of unital \(*\)-homomorphisms from \(A_n\) to \(M_D\) with the point-norm topology.  This is a compact space because it is a closed subset of the compact product
\[
        \prod_{a\in(A_n)_1}(M_D)_1.
\]
The \(D/r_n\)-fold sum of an irreducible representation of \(A_n\) of dimension \(r_n\) belongs to \(\operatorname{Rep}_D(A_n)\), so the latter is nonempty.

Set
\[
        K=\prod_{n=1}^{\infty}\operatorname{Rep}_D(A_n),
\]
and let
\[
        F_n=
        \big\{(\pi_j)_j\in K:\pi_n=\pi_{n+1}\circ\phi_n\big\}.
\]
Each \(F_n\) is closed.  The family \((F_n)_n\) has the finite intersection property: given \(N\), choose \(\rho\in\operatorname{Rep}_D(A_{N+1})\), set
\[
        \pi_j=\rho\circ\phi_{j,N+1},
        \ 1\leq j\leq N+1,
\]
where \(\phi_{N+1,N+1}\) is the identity, and choose only the
\(\pi_j\) with \(j>N+1\) arbitrarily.  This gives a point of
\(F_1\cap\cdots\cap F_N\).
Compactness of \(K\) therefore gives a compatible family
\[
        (\pi_n)_n\in\bigcap_{n=1}^{\infty}F_n.
\]
It induces a unital \(*\)-homomorphism \(\pi\colon A\longrightarrow M_D\).
Since \(A\) is simple, \(\pi\) is injective, forcing \(A\) to be finite-dimensional.  This contradicts the assumption that \(A\) is non-elementary, and proves that \(r_n\to\infty\).

Since \(\liminf_n q_n=0\), there is a sequence
\(n_1<n_2<\cdots\)
such that
\[
        q_{n_k}<\frac1k,
        \qquad k\in\N.
\]
Because \(r_n\to\infty\), we also have
\(r_{n_k}\to\infty\),
while \(q_{n_k}\to0\) by construction. Hence the telescoped system
with building blocks \(A_{n_k}\) is strictly subquadratic.
\end{proof}
\subsection{Proof of the AH theorem}

\begin{prop}
\label{prop:AH-local-subquadratic}
Let \(A=\varinjlim(A_n,\phi_n)\) be a unital AH algebra with a strictly subquadratic decomposition in the sense of Definition~\ref{def:AH-subquadratic}. Then \(A\) has unital locally tracially subquadratic homogeneous approximation in the sense of Definition~\ref{def:unital-local-subquadratic}.
\end{prop}

\begin{proof}
View each \(A_n\) as a unital subalgebra of \(A\). Let
\(\mathcal F\subset A_1\) be finite, and let
\(\varepsilon>0\), \(\eta>0\), and \(R\in\N\). Choose \(n\) large
enough that
\[
        r_n\geq R,
        \qquad
        q_n<\eta,
\]
and, for each \(a\in\mathcal F\), there is \(x_a\in A_n\) with
\[
        \|a-x_a\|<\varepsilon/2.
\]
Put
\[
        y_a=\frac{x_a}{\max\{1,\|x_a\|\}}\in(A_n)_1.
\]
Since \(\|a\|\leq1\),
\[
        \|x_a-y_a\|
        =
        \max\{0,\|x_a\|-1\}
        \leq\|x_a-a\|,
\]
and therefore
\[
        \|a-y_a\|
        \leq
        \|a-x_a\|+\|x_a-y_a\|
        <\varepsilon.
\]
Norm approximation implies
\(\|\cdot\|_{2,T(A)}\)-approximation, so \(C=A_n\), together with the
elements \(y_a\), satisfies all the requirements of
Definition~\ref{def:unital-local-subquadratic}.
\end{proof}

\begin{thm}
\label{thm:AH-gamma}
Let \(A\) be a separable, simple, unital, nuclear, non-elementary AH algebra with subquadratic growth in the sense of Definition~\ref{def:AH-subquadratic}. Then \(A\) has uniform property \(\Gamma\).
\end{thm}

\begin{proof}
Choose a subquadratic AH decomposition of \(A\).  Every stage has a
tracial state, and compactness of the product of the stage trace spaces,
together with the finite-intersection property for the compatibility
conditions, gives \(T(A)\neq\emptyset\).  By
Proposition~\ref{prop:AH-subquadratic-to-strict}, after telescoping this
decomposition is strictly subquadratic.
Proposition~\ref{prop:AH-local-subquadratic} then gives unital locally
tracially subquadratic homogeneous approximation, and
Theorem~\ref{thm:local-gamma} gives uniform property \(\Gamma\).
\end{proof}

\section{Nonunital tracial local approximation}
\label{sec:vaccaro-local}

\subsection{Locally tracially subquadratic homogeneous approximation}

Definition~\ref{def:unital-local-subquadratic} requires the local
homogeneous model to share its unit with the ambient algebra and this provides a clean route to Theorem~\ref{thm:local-gamma}, the natural counterpoint to the main result of \cite{TomsSchubCalc:preprint}.  We now give a more flexible definition for tracial local subquadratic approximation which permits nonunital local models.  The ambient algebra remains
unital, but our definition can be readily adapted to a nonunital ambient algebra should applications require it.  

\begin{notation}
\label{notation:one-point-compactification}
For a locally compact metrizable space \(Z\) we set \(\alpha Z = Z\) when \(Z\) is compact, and otherwise declare \(\alpha Z\) to be the one-point compactification of \(Z\).
\end{notation}

\begin{defn}
\label{def:vaccaro-local-subquadratic}
Let \(A\) be a unital \(C^*\)-algebra with \(T(A)\neq\emptyset\).  We
say that \(A\) has \emph{locally tracially subquadratic homogeneous
approximation} if, for every finite subset \(\mathcal F\subset A_1\),
every \(\varepsilon>0\), every \(\eta>0\), and every \(R\in\N\), there
is a \(C^*\)-subalgebra \(C\subseteq A\) of the form
\[
        C=\bigoplus_{i=1}^m B_i,
        \qquad
        B_i=\Gamma_0(\operatorname{End}(E_i)),
\]
where \(Z_i\) is locally compact and metrizable, \(E_i\to Z_i\) is the
restriction of a rank-\(r_i\) complex vector bundle over \(\alpha Z_i\),
and \(\alpha Z_i\) has finite covering dimension.  We further require that
\begin{enumerate}[(1)]
\item every element of \(\mathcal F\) is within \(\varepsilon\) in
      \(\|\cdot\|_{2,T(A)}\) of an element of \(C_1\);
\item \(r_i\geq R\) for every \(i\), and
\[
        \max_{1\leq i\leq m}
        \frac{\dim(\alpha Z_i)}{r_i^2}<\eta.
\]
\end{enumerate}
\end{defn}

Definition~\ref{def:vaccaro-local-subquadratic} is the direct
subquadratic analogue of Vaccaro's tracially LTH condition with locally
flat dimension growth from \cite[Definition~3.1]{Vaccaro:preprint}.
Indeed, suppose that for fixed \(\mathcal F\) and \(\varepsilon\),
Vaccaro's definition provides a constant \(N>0\) and, for every
\(k\in\N\), an approximating algebra
\[
        \bigoplus_i M_{n_i}(C_0(Z_i))
\]
with \(n_i\geq k\) and \(\dim(Z_i)/n_i<N\).  Given \(\eta>0\) and
\(R\in\N\), choose \(k\geq R\) so large that \(N/k<\eta\).  Using
\(\dim(\alpha Z_i)\leq\dim(Z_i)\), as in the proof of
\cite[Lemma~3.4]{Vaccaro:preprint}, we obtain
\[
        \frac{\dim(\alpha Z_i)}{n_i^2}
        \leq
        \frac{\dim(Z_i)}{n_i^2}
        <\frac{N}{n_i}
        \leq\frac{N}{k}
        <\eta.
\]
Thus, Vaccaro's hypothesis implies
Definition~\ref{def:vaccaro-local-subquadratic}.  Theorem~\ref{thm:vaccaro-local-gamma} below therefore recovers the uniform
property~\(\Gamma\) conclusion of \cite[Theorem~3.6]{Vaccaro:preprint}, extending it from locally flat to arbitrary subquadratic dimension
growth.  The main result of \cite{TomsSchubCalc:preprint} shows that the quadratic endpoint cannot be included in general so that Theorem~\ref{thm:vaccaro-local-gamma} is optimal.  Note that
Definition~\ref{def:unital-local-subquadratic} is the special case of Definition~\ref{def:vaccaro-local-subquadratic}
in which the local model shares a unit with the ambient algebra and the base spaces \(Z_i\) are compact.

If \(C\subseteq A\) is a \(C^*\)-subalgebra, let \(s_C\in A^{**}\)
denote the image of the unit of \(C^{**}\) under the bidual of the
inclusion.  If \(C=\bigoplus_{i=1}^m B_i\), write \(s_i\) for the image
of the unit of \(B_i^{**}\), so that
\(s_C=s_1+\cdots+s_m\) orthogonally.  We identify a bounded trace on a
possibly nonunital \(C^*\)-algebra with its normal extension to the
bidual, and hence with its restriction to the multiplier algebra.
Accordingly, if \(x=x^*\in B_i\), then
\(\eta_\delta(x-\lambda)\) means
\[
        \eta_\delta\bigl(x-\lambda 1_{M(B_i)}\bigr)
        \in M(B_i).
\]

\subsection{Cutoffs and rank avoidance}

\begin{lem}
\label{lem:unit-cutoff}
Let \(A\) be unital, let
\[
        C=\bigoplus_{i=1}^m B_i\subseteq A,
        \qquad
        B_i=\Gamma_0(\operatorname{End}(E_i)),
\]
and let \(e\in C_1\). Put
\[
        u=e^*e=\bigoplus_{i=1}^m u_i.
\]
There are compactly supported central positive contractions
\[
        q_i\in C_c(Z_i)\cdot1_{E_i}\subseteq B_i
\]
such that, with \(q=\bigoplus_iq_i\), the following inequality holds
in \(A^{**}\):
\[
        1_A-q\leq4(1_A-u)^2.
\]
Consequently,
\[
\begin{aligned}
        \sup_{\tau\in T(A)}\tau(1_A-q)
        &\leq4\|1_A-u\|_{2,T(A)}^2 \\
        &\leq16\|1_A-e\|_{2,T(A)}^2.
\end{aligned}
\]
\end{lem}

\begin{proof}
Let \(s_i\in A^{**}\) be the support projection of \(B_i\), and put
\[
        s_C=\sum_{i=1}^m s_i.
\]
For each \(i\), the set
\[
        K_i=\{z\in Z_i:\|u_i(z)\|\geq1/2\}
\]
is compact because \(u_i\) vanishes at infinity. Choose
\(q_i\in C_c(Z_i)_+\) with \(0\leq q_i\leq1\) and \(q_i=1\) on
\(K_i\), and regard \(q_i\) as the corresponding central element of
\(B_i\).

If \(q_i(z)<1\), then \(\|u_i(z)\|<1/2\), and hence
\[
        (1-q_i(z))1_{E_{i,z}}
        \leq1_{E_{i,z}}
        \leq4(1_{E_{i,z}}-u_i(z))^2.
\]
The same inequality is trivial when \(q_i(z)=1\). Thus, under the
canonical embedding of \(M(B_i)\) into the corner
\(s_iA^{**}s_i\),
\[
        s_i-q_i\leq4(s_i-u_i)^2.
\]

The projections
\[
        1_A-s_C,s_1,\ldots,s_m
\]
are mutually orthogonal, and \(u=\sum_i u_i\). Therefore, in \(A^{**}\),
\[
\begin{aligned}
        1_A-q
        &=(1_A-s_C)+\sum_{i=1}^m(s_i-q_i) \\
        &\leq
        4(1_A-s_C)+4\sum_{i=1}^m(s_i-u_i)^2 \\
        &=4(1_A-u)^2.
\end{aligned}
\]
Applying the normal extension of a trace \(\tau\in T(A)\) gives the
first estimate.

Finally,
\[
        1_A-e^*e=(1_A-e^*)+e^*(1_A-e),
\]
so the tracial property and \(\|e\|\leq1\) give
\[
        \|1_A-e^*e\|_{2,T(A)}
        \leq2\|1_A-e\|_{2,T(A)}.
\]
This yields the second estimate.
\end{proof}

\begin{lem}
\label{lem:locally-compact-rank-avoidance}
Let \(Z\) be locally compact and metrizable, suppose that
\(d=\dim(\alpha Z)<\infty\), and let \(E\to Z\) be the restriction of
a rank-\(r\) vector bundle over \(\alpha Z\).  Put
\(B=\Gamma_0(\operatorname{End}(E))\), let
\(q\in C_c(Z)_+\cdot1_E\) be a central positive contraction, and let
\(\Lambda\subset(-1,1)\) be finite.  If \(1\leq k\leq r\) and
\(d<k^2\), then, for every \(f\in B_{1,\mathrm{sa}}\) and every
\(\varepsilon>0\), there are \(g\in B_{1,\mathrm{sa}}\) and
\(\delta>0\) such that \(\|f-g\|<\varepsilon\) and
\[
        \tr_r\!\left(
        \eta_\delta(g(z)-\lambda1_{E_z})
        \right)
        \leq \frac{k}{r}q(z)+1-q(z),
        \qquad z\in Z,\ \lambda\in\Lambda,
\]
where \(\tr_r\) denotes the normalized fiber trace.  It follows that for
every bounded positive trace \(\rho\) on \(B\),
\[
        \rho\bigl(\eta_\delta(g-\lambda)\bigr)
        \leq
        \frac{k}{r}\rho(q)
        +\rho\bigl(1_{M(B)}-q\bigr),
        \qquad \lambda\in\Lambda.
\]
\end{lem}

\begin{proof}
If \(\Lambda=\emptyset\), take \(g=f\) and, for example, \(\delta=1\).
Assume henceforth that \(\Lambda\neq\emptyset\).  Let
\(\widetilde E\to\alpha Z\) be the compactification bundle, i.e.,
\(\widetilde E|_Z=E\).  Extend \(f\) to
\(\widetilde f\in\Gamma(\operatorname{End}(\widetilde E))_{1,\mathrm{sa}}\)
by setting \(\widetilde f(\infty)=0\).  By
Proposition~\ref{prop:compact-bundle-rank-avoidance}, there is
\(h\in\Gamma(\operatorname{End}(\widetilde E))_{1,\mathrm{sa}}\) such
that
\[
        \|\widetilde f-h\|<\varepsilon/2
        \quad\text{and}\quad
        \dim\ker(h(z)-\lambda1_{\widetilde E_z})<k
\]
for every \(z\in\alpha Z\) and \(\lambda\in\Lambda\).  By compactness
of \(\alpha Z\), finiteness of \(\Lambda\), and the argument in
Corollary~\ref{cor:compact-tracial-rank-avoidance}, there is
\(\delta>0\) such that
\[
        \tr_r\!\left(
        \eta_\delta(h(z)-\lambda1_{\widetilde E_z})
        \right)<\frac{k}{r}
\]
for every \(z\in\alpha Z\) and \(\lambda\in\Lambda\).

Choose \(\chi\in C_c(Z)_+\) with \(0\leq\chi\leq1\), with
\(\chi=1\) on \(\supp(q)\), and with
\(\|f-\chi f\|<\varepsilon/2\).  Set
\(g=\chi h|_Z\).
Then \(g\in B_{1,\mathrm{sa}}\) and
\[
        \|f-g\|
        \leq\|f-\chi f\|+\|\chi(f-h|_Z)\|
        <\varepsilon.
\]
If \(q(z)>0\), then \(\chi(z)=1\), so \(g(z)=h(z)\) and
\[
        \tr_r\!\left(
        \eta_\delta(g(z)-\lambda1_{E_z})
        \right)<\frac{k}{r}
        \leq \frac{k}{r}q(z)+1-q(z).
\]
If \(q(z)=0\), the desired inequality follows from
\(0\leq\eta_\delta\leq1\).  This proves the fiberwise estimate. If \(\rho\) is a bounded positive
trace on \(B\), then its restriction to the central copy of
\(C_0(Z)\) is represented by a finite positive Radon measure
\(\mu_\rho\). Since each matrix fiber has a unique normalized trace,
local trivializations and a partition of unity give
\[
        \rho(v)
        =
        \int_Z\tr_r(v(z))\,d\mu_\rho(z),
        \qquad v\in B.
\]
The same formula holds for the normal extension of \(\rho\) on the
positive multiplier elements appearing here. Integrating the
fiberwise estimate therefore gives
\[
        \rho\bigl(\eta_\delta(g-\lambda)\bigr)
        \leq
        \frac{k}{r}\rho(q)
        +\rho\bigl(1_{M(B)}-q\bigr),
        \qquad \lambda\in\Lambda.
\]
\end{proof}

\subsection{The local theorem}

\begin{thm}
\label{thm:vaccaro-local-rr0}
Let \(A\) be a unital \(C^*\)-algebra with \(T(A)\neq\emptyset\).  If
\(A\) has locally tracially subquadratic homogeneous approximation,
then \(A^{\U}\) has real rank zero.
\end{thm}

\begin{proof}
We verify property \((FG)\).  Let \(a_n\in A_{1,\mathrm{sa}}\) be a
sequence, and let \(\Lambda\subset(-1,1)\) be finite.  Choose
\(\theta_n\in(0,1/4)\) with \(\theta_n\to0\), and put
\[
        \varepsilon_n
        =\min\left\{\frac1n,\frac{\sqrt{\theta_n}}4\right\},
        \qquad
        \eta_n=\frac{\theta_n^2}{4}.
\]
Choose \(R_n\in\N\) such that \(1/R_n<\theta_n/4\).  Apply
Definition~\ref{def:vaccaro-local-subquadratic} to
\(\{a_n,1_A\}\), with approximation tolerance \(\varepsilon_n\),
dimension-rank tolerance \(\eta_n\), and rank bound \(R_n\).  We
obtain
\[
        C_n=\bigoplus_{i=1}^{m_n}B_{n,i},
        \qquad
        B_{n,i}=\Gamma_0(\operatorname{End}(E_{n,i})),
\]
and contractions \(d_n,e_n\in C_n\) such that
\[
        \|a_n-d_n\|_{2,T(A)}<\varepsilon_n,
        \qquad
        \|1_A-e_n\|_{2,T(A)}<\varepsilon_n.
\]
Set
\[
        c_n=\frac{d_n+d_n^*}{2}
            =\bigoplus_{i=1}^{m_n}c_{n,i}.
\]
Then \(c_n\in(C_n)_{1,\mathrm{sa}}\) and
\(\|a_n-c_n\|_{2,T(A)}<\varepsilon_n\).

Put \(u_n=e_n^*e_n\).  Lemma~\ref{lem:unit-cutoff} gives compactly
supported central positive contractions \(q_{n,i}\in B_{n,i}\) such
that, for \(q_n=\bigoplus_i q_{n,i}\),
\[
        \sup_{\tau\in T(A)}\tau(1_A-q_n)
        \leq16\varepsilon_n^2
        \leq\theta_n.
\]
Let \(r_{n,i}=\rank(E_{n,i})\), and set
\[
        k_{n,i}=\left\lceil\frac{\theta_n r_{n,i}}2\right\rceil.
\]
Since \(r_{n,i}\geq R_n\),
\[
        \frac{k_{n,i}}{r_{n,i}}
        \leq\frac{\theta_n}{2}+\frac1{r_{n,i}}
        <\theta_n.
\]
In particular, \(1\leq k_{n,i}\leq r_{n,i}\).  Moreover,
\[
        \dim(\alpha Z_{n,i})
        <\frac{\theta_n^2r_{n,i}^2}{4}
        \leq k_{n,i}^2.
\]
Apply Lemma~\ref{lem:locally-compact-rank-avoidance} to
\(c_{n,i}\), the cutoff \(q_{n,i}\), the set \(\Lambda\), the integer
\(k_{n,i}\), and the tolerance \(\varepsilon_n\).  We obtain
\(b_{n,i}\in(B_{n,i})_{1,\mathrm{sa}}\) and \(\delta_{n,i}>0\) with
\(\|b_{n,i}-c_{n,i}\|<\varepsilon_n\)
and, for every bounded positive trace \(\rho\) on \(B_{n,i}\),
\[
        \rho\bigl(\eta_{\delta_{n,i}}(b_{n,i}-\lambda)\bigr)
        \leq
        \frac{k_{n,i}}{r_{n,i}}\rho(q_{n,i})
        +\rho\bigl(1_{M(B_{n,i})}-q_{n,i}\bigr),
        \qquad \lambda\in\Lambda.
\]
Set
\[
        b_n=\bigoplus_{i=1}^{m_n}b_{n,i},
        \qquad
        \delta_n=\min_i\delta_{n,i}.
\]
Then
\(\|a_n-b_n\|_{2,T(A)}<2\varepsilon_n\),
so \((a_n-b_n)_n\in c_{\U}(A)\).

Fix \(\tau\in T(A)\) and \(\lambda\in\Lambda\), and let \(s_{n,i}\)
and \(s_{C_n}\) be the support projections introduced above.  Since
\(\delta_n\leq\delta_{n,i}\), we have
\(\eta_{\delta_n}\leq\eta_{\delta_{n,i}}\).  Moreover,
\(b_{n,i}\in s_{n,i}A^{**}s_{n,i}\),
\(b_n=\sum_i b_{n,i}\), and
\(s_{C_n}=\sum_i s_{n,i}\).  Thus the mutually orthogonal projections
\(1-s_{C_n},s_{n,1},\ldots,s_{n,m_n}\) sum to \(1_{A^{**}}\) and
reduce \(b_n-\lambda 1_{A^{**}}\).  Writing
\(\eta_{\delta_n}^{(i)}\) for functional calculus in the unital corner
\(s_{n,i}A^{**}s_{n,i}\), we therefore have
\[
\begin{aligned}
 b_n-\lambda 1_{A^{**}}
 &=-\lambda(1-s_{C_n})
   +\sum_{i=1}^{m_n}(b_{n,i}-\lambda s_{n,i}),\\
 \eta_{\delta_n}(b_n-\lambda 1_{A^{**}})
 &=\eta_{\delta_n}(-\lambda)(1-s_{C_n})
   +\sum_{i=1}^{m_n}
     \eta_{\delta_n}^{(i)}(b_{n,i}-\lambda s_{n,i})\\
 &=\eta_{\delta_n}(-\lambda)(1-s_{C_n})
   +\sum_{i=1}^{m_n}\eta_{\delta_n}(b_{n,i}-\lambda).
\end{aligned}
\]
The second equality is the componentwise functional calculus for these
orthogonal corners, while the final equality uses the multiplier-algebra
convention fixed above for each summand \(B_{n,i}\).
Applying the preceding trace estimates to the bounded traces
\(\tau|_{B_{n,i}}\), and using
\(0\leq\eta_{\delta_n}(-\lambda)\leq1\), gives
\[
\begin{aligned}
        \tau\bigl(\eta_{\delta_n}(b_n-\lambda)\bigr)
        &\leq \tau(1-s_{C_n})
        +\sum_{i=1}^{m_n}
          \left(
          \frac{k_{n,i}}{r_{n,i}}\tau(q_{n,i})
          +\tau(s_{n,i}-q_{n,i})
          \right)\\
        &\leq \theta_n\tau(q_n)+\tau(1-q_n)\\
        &\leq2\theta_n.
\end{aligned}
\]
Thus
\[
        \max_{\lambda\in\Lambda}
        \sup_{\tau\in T(A)}
        \tau\bigl(\eta_{\delta_n}(b_n-\lambda)\bigr)
        \leq2\theta_n\longrightarrow0.
\]
Thus, we have established property \((FG)\), so
\(\rr(A^{\U})=0\) by Proposition~\ref{prop:S-criterion}.
\end{proof}

\begin{lem}
\label{lem:vaccaro-local-tlfnd}
If \(A\) has locally tracially subquadratic homogeneous approximation,
then \(A\) has tracially locally finite nuclear dimension.
\end{lem}

\begin{proof}
By scaling, it is enough to consider finite subsets of \(A_1\). Apply
Definition~\ref{def:vaccaro-local-subquadratic} with arbitrary positive
dimension--rank tolerance and with rank bound \(R=1\).

Since \(E_i\) extends to a vector bundle over the compact space
\(\alpha Z_i\), there are \(N_i\in\N\) and a projection
\[
        p_i\in M_{N_i}(C(\alpha Z_i))
\]
such that
\[
        \Gamma_0(\operatorname{End}(E_i))
        \cong
        p_iM_{N_i}(C_0(Z_i))p_i.
\]
This is a hereditary subalgebra of \(M_{N_i}(C_0(Z_i))\). Since
\[
        \dim Z_i\leq\dim(\alpha Z_i)<\infty,
\]
\cite[Propositions~2.3--2.5]{WinterZacharias:AIM} shows that each
summand has finite nuclear dimension. Finite direct sums preserve
finite nuclear dimension, so the resulting algebra \(C\) gives the
required tracial local approximation.
\end{proof}

\begin{thm}
\label{thm:vaccaro-local-gamma}
Let \(A\) be separable, simple, unital, and non-elementary, with
\(T(A)\neq\emptyset\).  If \(A\) has locally tracially subquadratic
homogeneous approximation, then \(A\) has uniform property \(\Gamma\).
\end{thm}

\begin{proof}
Since \(A\) is simple and non-elementary, it has no nonzero
finite-dimensional representations.  By
Theorem~\ref{thm:vaccaro-local-rr0}, the uniform tracial ultrapower
\(A^{\U}\) has real rank zero, and by
Lemma~\ref{lem:vaccaro-local-tlfnd}, \(A\) has tracially locally finite
nuclear dimension.  Theorem~\ref{thm:VW-bridge} therefore gives
uniform property \(\Gamma\).
\end{proof}

\section{Unital local recursive subhomogeneous approximation}
\label{sec:rsh-local}

In this final section we prove that Theorem~\ref{thm:local-gamma} generalizes to the situation where our local models are now recursive subhomogeneous algebras (RSH) instead of merely homogeneous ones.  This class is significantly larger, witnessed for instance by the presence of algebras which contain no nontrivial projections.  We note that the clutching \(*\)-homomorphisms that form part of the defining data for an RSH algebra need not have any special form.  The only extra ingredient beyond the homogeneous proof is a relative form of the same rank-avoidance argument, applied one pullback block at a time.

\subsection{Recursive subhomogeneous decompositions}

We use the following concrete version of recursive subhomogeneous decomposition originally introduced in \cite{PhillipsRSH}.  A \emph{unital RSH decomposition} of length \(\ell\) is a sequence of C$^*$-algebras
\[
        R^{(0)},R^{(1)},\ldots,R^{(\ell)}
\]
constructed as follows.  First
\(R^{(0)}=C(X_0,M_{r_0})\)
for a compact metrizable finite-dimensional space \(X_0\).  Having constructed \(R^{(j-1)}\), choose a compact metrizable finite-dimensional space \(X_j\), a closed subset \(Y_j\subseteq X_j\), an integer \(r_j\geq 1\), and a unital \(*\)-homomorphism
\[
        \varphi_j\colon R^{(j-1)}\longrightarrow C(Y_j,M_{r_j}).
\]
Then set
\[
        R^{(j)}=
        \big\{(s,f)\in R^{(j-1)}\oplus C(X_j,M_{r_j}):
        \varphi_j(s)=f|_{Y_j}\big\}.
\]
The final algebra \(R=R^{(\ell)}\) is called a unital RSH algebra with
respect to this chosen decomposition.  We set \(Y_0=\emptyset\).  For
\(0\leq j\leq\ell\), the \(j\)-th \emph{block} of the decomposition is
the coordinate algebra \(C(X_j,M_{r_j})\), regarded together with the
associated triple \((X_j,Y_j,r_j)\).  Whenever we refer to
\(C(X_j,M_{r_j})\) as a block, this triple is understood to be part of
the block data.  For \(a\in R\), we write
\(a_j\in C(X_j,M_{r_j})\) for its coordinate on the \(j\)-th block.
Throughout this section, every finite-dimensional representation of a
unital algebra is understood to be unital, equivalently nondegenerate.

The estimates below are decomposition-dependent.  This is deliberate: the local approximation hypothesis supplies the approximating RSH algebras together with decompositions for which the dimension estimate holds.

\begin{defn}
\label{def:blockwise-rsh-subquadratic}
Let \(0<\theta<1\).  A chosen unital RSH decomposition satisfies the \emph{\(\theta\)-subquadratic block estimate} if
\[
        \dim X_j < \lceil \theta r_j\rceil^2
\]
for every block \(C(X_j,M_{r_j})\).
\end{defn}

\begin{rmk}
The inequality
\[
        \max_j \frac{\dim X_j}{r_j^2} < \theta^2
\]
gives \(\dim(X_j) < {\lceil\theta r_j\rceil}^2\) as required for the codimension estimate and does not
represent a stronger asymptotic hypothesis.
\end{rmk}

\begin{defn}
\label{def:rsh-subquadratic-growth}
A \emph{unital RSH inductive-limit decomposition} of a unital
\(C^*\)-algebra \(A\) is an isomorphism
\[
        A\cong\varinjlim(A_n,\phi_n),
\]
where the connecting maps are unital and injective and each \(A_n\)
is equipped with a chosen unital RSH decomposition.
Write
\[
        \big\{C(X_{n,j},M_{r_{n,j}}):0\leq j\leq \ell_n\big\}
\]
for the blocks in the chosen decomposition of \(A_n\), and set
\[
        q_n=
        \max_{0\leq j\leq \ell_n}
        \frac{\dim X_{n,j}}{r_{n,j}^2}.
\]
We say that the displayed inductive-limit decomposition has
\emph{subquadratic RSH dimension growth} if
\[
        \liminf_{n\to\infty}q_n=0.
\]
A unital \(C^*\)-algebra has subquadratic RSH dimension growth if it
admits such a decomposition.  For ASH algebras, we also refer to this
simply as \emph{subquadratic growth}.  Since \(Y_j=\emptyset\) gives
\[
R^{(j)}=R^{(j-1)}\oplus C(X_j,M_{r_j}),
\]
direct sums are inherently handled by this definition.
\end{defn}

\subsection{Relative rank avoidance}

Here we isolate the relative transversality result needed to handle pullbacks.

\begin{prop}
\label{prop:relative-compact-rank-avoidance}
Let \(Y\subseteq X\) be a closed subset of a compact metrizable space of finite covering dimension \(d\).  Let \(r\in\N\), let \(1\leq k\leq r\), and let \(\Lambda\subset(-1,1)\) be finite.  Suppose \(d<k^2\).  Let \(f\in C(X,M_r)_{\mathrm{sa}}\) satisfy
\[
        \dim\ker(f(y)-\lambda 1_r)<k,
        \qquad y\in Y,\ \lambda\in\Lambda.
\]
Then, for every \(\varepsilon>0\), there is \(g\in C(X,M_r)_{\mathrm{sa}}\) such that
\[
        g|_Y=f|_Y,
        \qquad \|g-f\|<\varepsilon,
\]
and
\[
        \dim\ker(g(x)-\lambda 1_r)<k,
        \qquad x\in X,\ \lambda\in\Lambda.
\]
\end{prop}

\begin{proof}
If \(\Lambda=\emptyset\), take \(g=f\). Otherwise, for
\(\lambda\in\Lambda\) and \(1\leq j\leq r\), put
\[
        D_{\lambda,j}
        =
        \{a\in M_r(\C)_{\mathrm{sa}}:
        \dim\ker(a-\lambda1_r)\geq j\},
\]
and
\[
        S_{\lambda,j}
        =
        D_{\lambda,j}\setminus D_{\lambda,j+1},
        \qquad
        D_{\lambda,r+1}=\emptyset.
\]
Enumerate
\[
        \Lambda=\{\lambda_1,\ldots,\lambda_L\},
\]
put \(N=L(r-k+1)\), and use the same order of exact-nullity strata as
in the proof of
Proposition~\ref{prop:compact-bundle-rank-avoidance}.

At the step \((\ell,j)\), the current section takes values in the open
manifold
\[
        V_{\ell,j}
        =
        M_r(\C)_{\mathrm{sa}}
        \setminus
        \left(
        \bigcup_{h<\ell}D_{\lambda_h,k}
        \cup D_{\lambda_\ell,j+1}
        \right).
\]
The current bad set
\[
        S_{\lambda_\ell,j}\cap V_{\ell,j}
\]
is closed and smooth in \(V_{\ell,j}\), with Lipschitz codimension
\[
        j^2\geq k^2>d.
\]

Every open subset \(U\subseteq X\) is metrizable and normal and has
covering dimension at most \(d\). Hence
\cite[Theorem~3.4]{JacobDimAvoidance} makes the current bad set
\(U\)-avoidable. Moreover, compact metrizability makes \(X\)
paracompact and perfectly normal. The relative clause of
\cite[Theorem~5.1]{JacobDimAvoidance} therefore gives a perturbation
which avoids the current stratum and agrees with the current section
on \(Y\).

This relative clause is applicable at every step: the original
boundary values avoid every locus \(D_{\lambda,k}\), and all preceding
perturbations have been fixed on \(Y\). Taking the perturbation at each
of the \(N\) steps to have norm less than
\(\varepsilon/(2N)\) yields a final section
\(g\in C(X,M_r)_{\mathrm{sa}}\) satisfying
\[
        g|_Y=f|_Y,
        \qquad
        \|g-f\|<\varepsilon,
\]
and
\[
        \dim\ker(g(x)-\lambda1_r)<k,
        \qquad
        x\in X,\ \lambda\in\Lambda.
\]
\end{proof}

We shall also use the following easy clipping estimate.

\begin{lem}
\label{lem:clipping-rsh}
Let \(\chi\colon\R\to[-1,1]\) be \(\chi(t)=\max\{-1,\min\{t,1\}\}\).  If \(a,b\in M_r(\C)_{\mathrm{sa}}\), \(\|b\|\leq 1\), and \(\|a-b\|\leq \gamma\), then
\[
        \|a-\chi(a)\|\leq \gamma,
        \qquad
        \|\chi(a)-b\|\leq 2\gamma.
\]
Moreover, if \(\lambda\in(-1,1)\), then
\[
        \ker(\chi(a)-\lambda 1_r)=\ker(a-
        \lambda 1_r).
\]
The same assertions hold fiberwise for self-adjoint sections of a trivial matrix bundle.
\end{lem}

\begin{proof}
Since \(\|b\|\leq1\) and \(\|a-b\|\leq\gamma\), the spectrum of \(a\) is contained in \([-1-\gamma,1+\gamma]\).  Hence \(|t-\chi(t)|\leq\gamma\) on \(\sigma(a)\), and functional calculus gives \(\|a-\chi(a)\|\leq\gamma\).  The second estimate follows from the triangle inequality.  Finally, for \(\lambda\in(-1,1)\), the scalar equality \(\chi(t)=\lambda\) holds if and only if \(t=\lambda\), so the kernel identity follows from functional calculus.
\end{proof}

\subsection{Rank avoidance in a unital RSH algebra}

The next theorem is the RSH replacement for the homogeneous rank-avoidance step.  The conclusion is stated for all finite-dimensional representations because this is exactly what is needed at the boundary of the next pullback block.

\begin{thm}
\label{thm:rsh-rank-avoidance}
Let \(R\) be a unital RSH algebra with a chosen decomposition, and let \(\Lambda\subset(-1,1)\) be finite.  Fix \(0<\theta<1\), and suppose that the decomposition satisfies
\[
        \dim X_j < \lceil \theta r_j\rceil^2
\]
for every block \(C(X_j,M_{r_j})\).  Then, for every \(a\in R_{1,\mathrm{sa}}\) and every \(\varepsilon>0\), there exists \(b\in R_{1,\mathrm{sa}}\) such that \(\|a-b\|<\varepsilon\) and
\[
        \dim\ker(\pi(b)-\lambda 1_{\dim\pi})
        < \theta\dim\pi
\]
for every finite-dimensional representation \(\pi\) of \(R\) and every \(\lambda\in\Lambda\).  In particular,
\[
        \dim\ker(b_j(x)-\lambda 1_{r_j})
        < \lceil \theta r_j\rceil,
        \qquad x\in X_j,\ \lambda\in\Lambda,
\]
for every block.
\end{thm}

\begin{proof}
We prove the theorem by induction on the length of the chosen RSH decomposition.

First suppose \(R=C(X,M_r)\).  Put \(k=\lceil\theta r\rceil\) and let \(a \in C(X,M_r)_{1,\mathrm{sa}}\) be given.  By hypothesis, \(\dim X<k^2\).  Proposition~\ref{prop:relative-compact-rank-avoidance}, with \(Y=\emptyset\), gives \(c\in C(X,M_r)_{\mathrm{sa}}\) with \(\|a-c\|<\varepsilon/3\) and
\[
        \dim\ker(c(x)-\lambda1_r)<k,
        \qquad x\in X,\ \lambda\in\Lambda.
\]
Since \(a\) is a contraction, Lemma~\ref{lem:clipping-rsh} gives \(b=\chi(c)\in R_{1,\mathrm{sa}}\), with \(\|a-b\|<\varepsilon\).  Note that the clipping estimate of Lemma~\ref{lem:clipping-rsh} preserves the kernel size estimate above.  Since the kernel dimension is an integer and is strictly smaller than \(k=\lceil\theta r\rceil\), it is at most \(k-1<\theta r\).

Every finite-dimensional representation of \(C(X,M_r)\) is a finite direct sum of point evaluations.  The preceding estimate is therefore additive over the summands and gives
\[
        \dim\ker(\pi(b)-\lambda1_{\dim\pi})<\theta\dim\pi
\]
for every finite-dimensional representation \(\pi\).

Now assume the result for decompositions of length \(\ell-1\), and write the last pullback as
\[
        R=S\oplus_{C(Y,M_r)}C(X,M_r),
\]
where \(S\) has length \(\ell-1\), the attaching map is \(\varphi\colon S\to C(Y,M_r)\), and \(\dim X<\lceil\theta r\rceil^2\).  Let \(a=(s,h)\in R_{1,\mathrm{sa}}\) so that \(h|_Y = \varphi(s)\), and fix \(\gamma < \varepsilon/4\).

By the induction hypothesis, there is \(s'\in S_{1,\mathrm{sa}}\) such that \(\|s-s'\|<\gamma\) and
\[
        \dim\ker(\rho(s')-\lambda1_{\dim\rho})
        <\theta\dim\rho
\]
for every finite-dimensional representation \(\rho\) of \(S\) and every \(\lambda\in\Lambda\).  On \(Y\), the boundary difference
\[
        \varphi(s'-s)\in C(Y,M_r)_{\mathrm{sa}}
\]
has norm at most \(\gamma\).  Extend its real matrix-coordinate functions to \(X\) by the scalar
Tietze extension theorem, obtaining
\(\widetilde d\in C(X,M_r)_{\mathrm{sa}}\).  Composing pointwise with
the radial retraction of the real vector space \(M_r(\C)_{\mathrm{sa}}\)
onto its closed ball of radius \(\gamma\) gives an extension
\(d\in C(X,M_r)_{\mathrm{sa}}\) with \(\|d\|\leq\gamma\); the radial
retraction fixes the prescribed boundary values.  Put
\(h_0=h+d\).
Then \(h_0|_Y=\varphi(s')\) and \(\|h_0-h\|\leq\gamma\).

Set \(k=\lceil\theta r\rceil\).  For each \(y\in Y\), the map
\(\operatorname{ev}_y\circ\varphi\colon S\to M_r\) is an
\(r\)-dimensional representation of \(S\).  Hence, for every
\(\lambda\in\Lambda\), the induction hypothesis gives
\[
        \dim\ker(h_0(y)-\lambda1_r)
        =
        \dim\ker\bigl((\operatorname{ev}_y\circ\varphi)(s')
        -\lambda1_r\bigr)
        <\theta r\leq k.
\]
Thus \(h_0|_Y\) satisfies the boundary hypothesis of
Proposition~\ref{prop:relative-compact-rank-avoidance}.

Apply Proposition~\ref{prop:relative-compact-rank-avoidance} to \((X,Y)\), the section \(h_0\), the integer \(k\), and tolerance \(\gamma\).  We obtain \(h_1\in C(X,M_r)_{\mathrm{sa}}\) such that
\[
        h_1|_Y=h_0|_Y=\varphi(s'),
        \qquad \|h_1-h_0\|<\gamma,
\]
and
\[
        \dim\ker(h_1(x)-\lambda1_r)<k,
        \qquad x\in X,\ \lambda\in\Lambda.
\]
Now set \(h'=\chi(h_1)\).  Since \(h_1|_Y=\varphi(s')\) and \(\varphi(s')\) is a self-adjoint contraction, the clipping function fixes the boundary values, so
\(h'|_Y=\varphi(s')\).
Also \(h'\) is a self-adjoint contraction.  By Lemma~\ref{lem:clipping-rsh}, the kernel estimates at the levels \(\lambda\in\Lambda\) are unchanged.  Moreover, \(\|h_1-h\|<2\gamma\), and the same lemma gives \(\|h'-h\|<4\gamma\).  Since \(\gamma<\varepsilon/4\), the element
\[
        b=(s',h')\in S\oplus_{C(Y,M_r)}C(X,M_r)=R
\]
is a self-adjoint contraction and satisfies \(\|a-b\|<\varepsilon\).

It remains only to check the representation estimate for \(R\). The map
\[
        \kappa\colon R\longrightarrow S,
        \qquad
        \kappa(s_0,h_0)=s_0,
\]
is surjective by the Tietze extension theorem, and
\[
        \ker(\kappa)
        =
        0\oplus C_0(X\setminus Y,M_r).
\]
Consequently, an irreducible representation of \(R\) either annihilates
this ideal, in which case it has the form
\(\rho\circ\kappa\)
for an irreducible representation \(\rho\) of \(S\), or is unitarily
equivalent to a point evaluation
\[
        (s_0,h_0)\longmapsto h_0(x)
\]
at some \(x\in X\setminus Y\). This is the standard ideal--quotient
description of the irreducible representations of an RSH pullback
\cite{PhillipsRSH}.

The first class satisfies the desired estimate by the induction
hypothesis. For the second class,
\[
        \dim\ker(h'(x)-\lambda1_r)
        \leq\lceil\theta r\rceil-1
        <\theta r.
\]
Now let \(\pi\) be an arbitrary finite-dimensional representation of \(R\).
Write, up to unitary equivalence,
\[
        \pi\cong \pi_1\oplus\cdots\oplus\pi_m
\]
with the \(\pi_i\) irreducible.  The preceding two cases give
\[
        \dim\ker(\pi_i(b)-\lambda1_{\dim\pi_i})
        <\theta\dim\pi_i,
        \qquad i=1,\ldots,m.
\]
Since kernels of block diagonal operators are direct sums of the kernels of
the blocks,
\[
\begin{aligned}
        \dim\ker(\pi(b)-\lambda1_{\dim\pi})
        &=\sum_{i=1}^m
          \dim\ker(\pi_i(b)-\lambda1_{\dim\pi_i})\\
        &< \theta\sum_{i=1}^m\dim\pi_i
         = \theta\dim\pi.
\end{aligned}
\]
This proves the induction step and hence the theorem.

\end{proof}

\begin{cor}
\label{cor:rsh-tracial-smallness}
Let \(R\), \(\Lambda\), and \(\theta\) be as in
Theorem~\ref{thm:rsh-rank-avoidance}.  Then, for every
\(a\in R_{1,\mathrm{sa}}\) and every \(\varepsilon>0\), there are
\(b\in R_{1,\mathrm{sa}}\) and \(\delta>0\) such that
\(\|a-b\|<\varepsilon\) and
\[
        \rho\big(\eta_\delta(b-\lambda)\big)<\theta,
        \qquad \rho\in T(R),\ \lambda\in\Lambda.
\]
\end{cor}

\begin{proof}
Apply Theorem~\ref{thm:rsh-rank-avoidance} to obtain \(b\) with the
stated norm estimate and finite-dimensional representation estimate.
If \(\Lambda=\emptyset\), any \(\delta>0\) finishes the proof.  For a
block \(C(X_j,M_{r_j})\), put \(k_j=\lceil\theta r_j\rceil\).  For fixed
\(j\) and \(\lambda\), the \(k_j\)-th eigenvalue of
\(|b_j(x)-\lambda1_{r_j}|\), listed in nondecreasing order, is a
strictly positive continuous function of \(x\in X_j\).  Since there
are only finitely many blocks and finitely many \(\lambda\)'s, we may
choose \(\delta>0\) such that
\[
        \rank\big(\eta_\delta(b_j(x)-\lambda)\big)
        \leq k_j-1<\theta r_j
\]
for every block \(C(X_j,M_{r_j})\), every \(x\in X_j\), and every \(\lambda\in\Lambda\).

Put
\[
        \theta_0=\max_j\frac{k_j-1}{r_j}.
\]
Since there are only finitely many blocks and
\[
        \frac{k_j-1}{r_j}<\theta
\]
for every \(j\), we have \(\theta_0<\theta\). Thus every irreducible representation \(\pi\) of \(R\) satisfies
\[
        \tr_{\dim\pi}\big(\eta_\delta(\pi(b)-\lambda1_{\dim\pi})\big)
        \leq\theta_0,
        \qquad \lambda\in\Lambda.
\]
For fixed \(\lambda\in\Lambda\), the function
\[
        T(R)\longrightarrow\R,
        \qquad
        \rho\longmapsto
        \rho\big(\eta_\delta(b-\lambda)\big),
\]
is continuous and affine, so it attains its maximum at an extreme
tracial state.

If \(\rho\in\partial_eT(R)\), then \(\pi_\rho(R)''\) is a finite
factor. Since \(R\) is subhomogeneous, this factor is a matrix algebra.
Thus \(\rho\) is the normalized matrix trace composed with a
finite-dimensional factor representation of \(R\). The preceding
representation estimate gives
\[
        \rho\big(\eta_\delta(b-\lambda)\big)
        \leq\theta_0
\]
for every extreme trace. Taking the maximum over \(T(R)\) yields
\[
        \rho\big(\eta_\delta(b-\lambda)\big)
        \leq\theta_0<\theta,
        \qquad
        \rho\in T(R),\ \lambda\in\Lambda.
\]
This proves the corollary.
\end{proof}

\subsection{The local RSH theorem}

\begin{defn}
\label{def:unital-local-rsh-subquadratic}
Let \(A\) be a unital \(C^*\)-algebra with
\(T(A)\neq\emptyset\).  We say that \(A\) has
\emph{unital locally tracially subquadratic RSH approximation} if,
for every finite set \(\mathcal F\subset A_1\), every
\(\varepsilon>0\), and every \(0<\theta<1\), there is a unital
subalgebra \(C\subseteq A\), with \(1_C=1_A\), equipped with a
chosen unital RSH decomposition, such that
\begin{enumerate}[(1)]
\item every element of \(\mathcal F\) is within \(\varepsilon\) in
\(\|\cdot\|_{2,T(A)}\) of an element of \(C_1\);
\item the chosen decomposition of \(C\) satisfies the
\(\theta\)-subquadratic block estimate of
Definition~\ref{def:blockwise-rsh-subquadratic}.
\end{enumerate}
\end{defn}

\begin{thm}
\label{thm:local-rsh-rr0}
Let \(A\) be a unital \(C^*\)-algebra with \(T(A)\neq\emptyset\).  If \(A\) has unital locally tracially subquadratic RSH approximation, then \(A^{\U}\) has real rank zero.
\end{thm}

\begin{proof}
We verify property \((FG)\) from Definition~\ref{def:property-S}.  Let \(a_n\in A_{1,\mathrm{sa}}\) be a sequence, and let \(\Lambda\subset(-1,1)\) be finite.  Choose \(\theta_n\in(0,1)\) with \(\theta_n\to0\).

For each \(n\), apply
Definition~\ref{def:unital-local-rsh-subquadratic} to the finite set
\(\{a_n\}\), tolerance \(1/n\), and parameter \(\theta_n\).  We obtain
a unital RSH subalgebra \(C_n\subseteq A\), equipped with a chosen
decomposition satisfying the \(\theta_n\)-subquadratic block estimate,
and an element \(d_n\in(C_n)_1\) such that
\(\|a_n-d_n\|_{2,T(A)}<1/n\).
Put
\[
        c_n=\frac{d_n+d_n^*}{2}.
\]
Then \(c_n\in(C_n)_{1,\mathrm{sa}}\) and
\[
        \|a_n-c_n\|_{2,T(A)}
        \leq\|a_n-d_n\|_{2,T(A)}<1/n,
\]
since \(\|x^*\|_{2,T(A)}=\|x\|_{2,T(A)}\).
By Corollary~\ref{cor:rsh-tracial-smallness}, applied inside \(C_n\), there are \(b_n\in(C_n)_{1,\mathrm{sa}}\) and \(\delta_n>0\) such that \(\|b_n-c_n\|<1/n\) and
\[
        \rho\big(\eta_{\delta_n}(b_n-\lambda)\big)<\theta_n,
        \qquad \rho\in T(C_n),\ \lambda\in\Lambda.
\]
It follows that \(\|a_n-b_n\|_{2,T(A)}<2/n\), so \((a_n-b_n)_n\in c_{\U}(A)\).

For \(\tau\in T(A)\), the restriction \(\tau|_{C_n}\) is a tracial state on \(C_n\), because \(C_n\) is unital with the same unit as \(A\).  Hence
\[
        \tau\big(\eta_{\delta_n}(b_n-\lambda)\big)<\theta_n,
        \qquad \tau\in T(A),\ \lambda\in\Lambda.
\]
Thus
\[
        \max_{\lambda\in\Lambda}\sup_{\tau\in T(A)}
        \tau\big(\eta_{\delta_n}(b_n-\lambda)\big)
        \leq \theta_n\longrightarrow0.
\]
This gives property \((FG)\), hence \(\rr(A^{\U})=0\) by Proposition~\ref{prop:S-criterion}.
\end{proof}

\begin{lem}
\label{lem:local-rsh-tlfnd}
If \(A\) has unital locally tracially subquadratic RSH approximation, then \(A\) has tracially locally finite nuclear dimension.
\end{lem}

\begin{proof}
It is enough, by scaling, to consider finite sets in the unit ball.
Given a finite set \(\mathcal F\subset A_1\) and \(\varepsilon>0\),
apply Definition~\ref{def:unital-local-rsh-subquadratic} with, for
instance, \(\theta=1/2\). We obtain a unital RSH subalgebra
\(C\subseteq A\), equipped with a finite RSH decomposition over
finite-dimensional compact spaces, such that every
\(a\in\mathcal F\) is within \(\varepsilon\) in
\(\|\cdot\|_{2,T(A)}\) of an element of \(C_1\).

By \cite[Theorem~1.6]{WinterSubhomogeneous}, \(C\) has finite
decomposition rank. Hence \(C\) has finite nuclear dimension by
\cite[Remark~2.2(ii)]{WinterZacharias:AIM}. Thus \(C\), together with
the approximants supplied by condition~\((1)\) of
Definition~\ref{def:unital-local-rsh-subquadratic}, verifies
Definition~\ref{def:tlfnd}.
\end{proof}

\begin{thm}[Local RSH subquadratic growth gives uniform property
\(\Gamma\)]
\label{thm:local-rsh-gamma}
Let \(A\) be separable and unital, with \(T(A)\neq\emptyset\), and
suppose that \(A\) has no nonzero finite-dimensional representations.
If \(A\) has unital locally tracially subquadratic RSH approximation,
then \(A\) has uniform property \(\Gamma\).
\end{thm}

\begin{proof}
By Theorem~\ref{thm:local-rsh-rr0}, the uniform tracial ultrapower \(A^{\U}\) has real rank zero.  By Lemma~\ref{lem:local-rsh-tlfnd}, \(A\) has tracially locally finite nuclear dimension.  The Vaccaro--Winter bridge, Theorem~\ref{thm:VW-bridge}, then gives uniform property \(\Gamma\).
\end{proof}

\subsection{From inductive limits to the local hypothesis}

\begin{prop}
\label{prop:rsh-inductive-local}
Let \(A\cong\varinjlim(A_n,\phi_n)\) be a unital RSH inductive-limit decomposition with subquadratic RSH dimension growth in the sense of Definition~\ref{def:rsh-subquadratic-growth}.  If \(T(A)\neq\emptyset\), then \(A\) has unital locally tracially subquadratic RSH approximation in the sense of Definition~\ref{def:unital-local-rsh-subquadratic}.
\end{prop}

\begin{proof}
Identify each \(A_n\) with its image in \(A\).  Let \(\mathcal F\subset A_1\) be finite, let \(\varepsilon>0\), and let \(0<\theta<1\).  By density of \(\bigcup_n A_n\) in \(A\), there are \(N\in\N\) and, for each \(a\in\mathcal F\), an element \(x_a\in A_N\) such that
\(\|a-x_a\|<\varepsilon/3\).
Since \(\|a\|\leq1\), we have \(\|x_a\|<1+\varepsilon/3\).  Thus, after replacing \(x_a\) by
\[
        \frac{x_a}{\max\{1,\|x_a\|\}},
\]
we may assume that \(x_a\in(A_N)_1\) and still have
\(\|a-x_a\|<2\varepsilon/3<\varepsilon\).

Since \(\liminf_n q_n=0\), there is \(n\geq N\) such that
\(q_n<\theta^2\).
Regard each \(x_a\) as an element of \((A_n)_1\), and put \(C=A_n\subseteq A\).  Then every element of \(\mathcal F\) is within \(\varepsilon\) in norm, and hence in \(\|\cdot\|_{2,T(A)}\), of an element of \(C_1\).

Moreover, for every block
\(C(X_{n,j},M_{r_{n,j}})\) in the chosen decomposition of
\(C=A_n\), one has
\[
\begin{aligned}
        \dim X_{n,j}
        &\leq q_n r_{n,j}^2 \\
        &<\theta^2 r_{n,j}^2 \\
        &\leq \lceil\theta r_{n,j}\rceil^2.
\end{aligned}
\]
Thus \(C\) satisfies both conditions in
Definition~\ref{def:unital-local-rsh-subquadratic}.

\end{proof}

\begin{proof}[Proof of Theorem~\ref{ASHgamma}]

\(A\) is stably finite, unital, and nuclear and so \(T(A) \neq \emptyset\).
Proposition~\ref{prop:rsh-inductive-local} shows that \(A\) has unital
locally tracially subquadratic RSH approximation.  Since \(A\) has no
nonzero finite-dimensional representations,
Theorem~\ref{thm:local-rsh-gamma} gives uniform property \(\Gamma\).
\end{proof}

\bibliographystyle{alpha}
\bibliography{ref_copy}

@article{BrownPedersen:JFA,
author = {L.~G. Brown and G.~K. Pedersen},
title = {{$\mathrm C^*$}-algebras of real rank zero},
journal = {J. Funct. Anal.},
volume = {99},
number = {1},
pages = {131--149},
year = {1991}
}

@article{CETW:IMRN,
author = {J.~Castillejos and S.~Evington and A.~Tikuisis and S.~White},
title = {Uniform property {$\Gamma$}},
journal = {Int. Math. Res. Not. IMRN},
volume = {2022},
number = {13},
pages = {9864--9908},
year = {2022},
doi = {10.1093/imrn/rnaa282}
}

@article{CETWW:IM,
author = {J.~Castillejos and S.~Evington and A.~Tikuisis and S.~White and W.~Winter},
title = {Nuclear dimension of simple {$\mathrm C^*$}-algebras},
journal = {Invent. Math.},
volume = {224},
number = {1},
pages = {245--290},
year = {2021}
}

@misc{ElliottNiu25,
author = {G.~A. Elliott and Z.~Niu},
title = {On the small boundary property and {$\mathcal Z$}-absorption},
year = {2025},
eprint = {2504.03611},
archivePrefix = {arXiv},
primaryClass = {math.OA},
note = {arXiv:2504.03611}
}

@incollection{ElliottRordam:Abel,
author = {G.~A. Elliott and M.~R{\o}rdam},
title = {Perturbation of {H}ausdorff moment sequences, and an application to the theory of {$\mathrm C^*$}-algebras of real rank zero},
booktitle = {Operator {A}lgebras: {T}he {A}bel {S}ymposium 2004},
series = {Abel Symp.},
volume = {1},
pages = {97--115},
publisher = {Springer},
address = {Berlin},
year = {2006}
}

@article {JacobDimAvoidance,
    AUTHOR = {Jacob, Beno\^it},
     TITLE = {On perturbations of continuous maps},
   JOURNAL = {Canad. Math. Bull.},
  FJOURNAL = {Canadian Mathematical Bulletin. Bulletin Canadien de
              Math\'ematiques},
    VOLUME = {56},
      YEAR = {2013},
    NUMBER = {1},
     PAGES = {92--101},
      ISSN = {0008-4395,1496-4287},
   MRCLASS = {54F45},
  MRNUMBER = {3009408},
MRREVIEWER = {Ihor\ Z.\ Stasyuk},
       DOI = {10.4153/CMB-2011-158-8},
       URL = {https://doi.org/10.4153/CMB-2011-158-8},
}

@article{MurrayVonNeumann4,
author = {F.~J. Murray and J.~von Neumann},
title = {On rings of operators. {IV}},
journal = {Ann. of Math. (2)},
volume = {44},
pages = {716--808},
year = {1943}
}

@article {NgWinter,
    AUTHOR = {Ng, Ping Wong and Winter, Wilhelm},
     TITLE = {A note on subhomogeneous {$C^\ast$}-algebras},
   JOURNAL = {C. R. Math. Acad. Sci. Soc. R. Can.},
  FJOURNAL = {Comptes Rendus Math\'ematiques de l'Acad\'emie des Sciences.
              La Soci\'et\'e{} Royale du Canada. Mathematical Reports of the
              Academy of Science. The Royal Society of Canada},
    VOLUME = {28},
      YEAR = {2006},
    NUMBER = {3},
     PAGES = {91--96},
      ISSN = {0706-1994,2816-5810},
   MRCLASS = {46L05 (46L40)},
  MRNUMBER = {2310490},
}

@article {PhillipsRSH,
    AUTHOR = {Phillips, N. Christopher},
     TITLE = {Recursive subhomogeneous algebras},
   JOURNAL = {Trans. Amer. Math. Soc.},
  FJOURNAL = {Transactions of the American Mathematical Society},
    VOLUME = {359},
      YEAR = {2007},
    NUMBER = {10},
     PAGES = {4595--4623},
      ISSN = {0002-9947,1088-6850},
   MRCLASS = {46L80 (19A13 19B14 19K14 46L05 46L55)},
  MRNUMBER = {2320643},
MRREVIEWER = {Valentin\ Deaconu},
       DOI = {10.1090/S0002-9947-07-03850-0},
       URL = {https://doi.org/10.1090/S0002-9947-07-03850-0},
}

@article{Toms2005Independence,
  author  = {Toms, Andrew S.},
  title   = {On the independence of {$K$}-theory and stable rank for simple {$C^*$}-algebras},
  journal = {J. Reine Angew. Math.},
  volume  = {578},
  year    = {2005},
  pages   = {185--199},
  doi     = {10.1515/crll.2005.2005.578.185}
}

@article{Toms:Ann,
author = {A.~S. Toms},
title = {On the classification problem for nuclear {$\mathrm C^*$}-algebras},
journal = {Ann. of Math. (2)},
volume = {167},
number = {3},
pages = {1029--1044},
year = {2008}
}

@article{Toms:JFA,
author = {A.~S. Toms},
title = {Flat dimension growth for {$\mathrm C^*$}-algebras},
journal = {J. Funct. Anal.},
volume = {238},
number = {2},
pages = {678--708},
year = {2006}
}

@misc{TomsSchubCalc:preprint,
      title = {Schubert Calculus and uniform property {$\Gamma$}},
      author = {Andrew S. Toms},
      year = {2026},
      eprint = {2606.12188},
      archivePrefix = {arXiv},
      primaryClass = {math.OA},
      note = {arXiv:2606.12188}
}

@misc{Vaccaro:preprint,
author = {A.~Vaccaro},
title = {Stable rank one, tracial local homogeneity and uniform property {$\Gamma$}},
year = {2026},
eprint = {2604.24682},
archivePrefix = {arXiv},
primaryClass = {math.OA},
note = {arXiv:2604.24682}
}

@article{Villadsen:JAMS,
author = {J.~Villadsen},
title = {On the stable rank of simple {$\mathrm C^*$}-algebras},
journal = {J. Amer. Math. Soc.},
volume = {12},
number = {4},
pages = {1091--1102},
year = {1999}
}

@article{Villadsen:JFA,
author = {J.~Villadsen},
title = {Simple {$\mathrm C^*$}-algebras with perforation},
journal = {J. Funct. Anal.},
volume = {154},
number = {1},
pages = {110--116},
year = {1998}
}

@article{Winter:IM12,
author = {W.~Winter},
title = {Nuclear dimension and {$\mathcal Z$}-stability of pure {$\mathrm C^*$}-algebras},
journal = {Invent. Math.},
volume = {187},
number = {2},
pages = {259--342},
year = {2012}
}

@article{WinterZacharias:AIM,
author = {W.~Winter and J.~Zacharias},
title = {The nuclear dimension of {$\mathrm C^*$}-algebras},
journal = {Adv. Math.},
volume = {224},
number = {2},
pages = {461--498},
year = {2010},
doi = {10.1016/j.aim.2009.12.005}
}

@article{WinterSubhomogeneous,
author = {W.~Winter},
title = {Decomposition rank of subhomogeneous {$\mathrm C^*$}-algebras},
journal = {Proc. London Math. Soc. (3)},
volume = {89},
number = {2},
pages = {427--456},
year = {2004},
doi = {10.1112/S0024611504014716}
}
\end{document}